\documentclass[12pt,reqno]{amsart}

\usepackage{graphicx}
\usepackage{amsfonts}
\usepackage{amsmath}
\usepackage{amssymb}
\usepackage{fancyhdr}
\usepackage{indentfirst}
\usepackage{booktabs}
\usepackage{verbatim}
\usepackage{color}
\usepackage{amsthm}
\usepackage{subfigure}

\usepackage{titletoc}

\newcommand{\cs}{\zeta(s)}
\newcommand{\tcs}{\tilde{\zeta}(s)}

\newcommand{\rd}{\,\mathrm{d}}
\numberwithin{equation}{section}
\newtheorem{theorem}{Theorem}[section]
\newtheorem{lemma}[theorem]{Lemma}

\newtheorem{remark}[theorem]{Remark}

\def\bx{{\bf x}}

\def\bZ{{\bf Z} }

\def\cE{\mathcal{E}}

\def\sgn{\textnormal{sgn}}

\renewcommand{\geq}{\geqslant}
\renewcommand{\ge}{\geqslant}
\renewcommand{\leq}{\leqslant}
\renewcommand{\le}{\leqslant}
\def\calO{{\mathcal O}}

\usepackage[left=1in,right=1in,top=1in,bottom=1in]{geometry}
\usepackage[colorlinks=true,
            linkcolor=red,
            urlcolor=blue,
            citecolor=cyan]{hyperref}

\begin{document}

\title[Dynamics on a Curve with Pairwise Hyper-singular Repulsion ]
{Dynamics of Particles on a Curve with Pairwise Hyper-singular Repulsion }

\author{Douglas Hardin}
\address{Department of Mathematics, Vanderbilt University, Nashville, TN 37240}
\email{douglas.hardin@Vanderbilt.edu}
 
\author{Edward B. Saff}
\address{Department of Mathematics, Vanderbilt University, Nashville, TN 37240}
\email{Edward.B.Saff@Vanderbilt.edu}

\author{Ruiwen Shu}
\address{Department of Mathematics and Center for Scientific Computation and Mathematical Modeling (CSCAMM)\newline
University of Maryland, College Park MD 20742}
\email{rshu@cscamm.umd.edu}

\author{Eitan Tadmor}
\address{Department of Mathematics, Center for Scientific Computation and Mathematical Modeling (CSCAMM)\newline
and Institute for Physical Science \& Technology (IPST)\newline
University of Maryland, College Park MD 20742}
\email{tadmor@umd.edu}

\date{\today}



\thanks{\textbf{Acknowledgment.} The first two authors acknowledge support, in part, 
by the U. S. National Science Foundation under grant DMS-1516400. ET was supported by NSF grants DMS16-13911 and ONR grant N00014-1812465.}
\date{\today}

\begin{abstract}
We investigate the large time behavior of $N$ particles restricted to a smooth closed curve in $\mathbb{R}^d$ and subject to a gradient flow with respect to Euclidean hyper-singular repulsive Riesz $s$-energy with $s>1.$ We show that regardless of their initial positions, for all $N$ and time $t$ large,  their normalized Riesz $s$-energy will be close to the $N$-point minimal possible. Furthermore, the distribution of such particles will be close to uniform with respect to arclength measure along the curve.
\end{abstract}

\maketitle
\tableofcontents

\maketitle

\section{Introduction}

In this paper we consider the first-order $N$-particle model
\begin{equation}\label{eq_si}
\dot{ z}_i = -N^{-s}\sum_{j\ne i} \nabla W(\bx( z_i)-\bx( z_j)) \cdot \bx'( z_i),
\end{equation}
where the particles are interacting through the potential
\begin{equation}\label{Rie}
W(\bx) = W(|\bx|) = \frac{|\bx|^{-s}}{s},
\end{equation}
which is a power-law repulsion potential, assumed to be {\it hyper-singular}: $s>1$.
Here $\bx( z),\, z\in \mathbb{R}$ in $\mathbb{R}^d$ is a unit-length, smooth, closed, non-self-intersecting curve with 1-periodic arc-length parametrization; i.e., $|\bx'( z)|=1$ and $\bx(z+1)=\bx(z)$ for all $z\in \mathbb{R}$. The $N$-particle configuration  $\{\bx(z_i)\}_{i=1}^N$   is represented by the parameters $\bZ=(z_1,z_2,\ldots,z_N)$, where    $z_i=z_i(t)$ are real-valued functions of the time $t\in \mathbb{R}_{\ge 0}$ for $i=1,2,\ldots, N$.  The system \eqref{eq_si} can be rewritten as a gradient flow of the form
\begin{equation}\label{gradflow}
\dot{\bZ}=-N\nabla E(\bZ),
\end{equation}
for the energy
\begin{equation}\label{EDef}
  E=E(\bZ) := N^{-s-1}\sum_{1\le i < j \le N}W(\bx( z_i)-\bx( z_j)),
\end{equation}
which
 satisfies the energy dissipation
\begin{equation}\label{eq0E}
\dot{E} = \nabla E(\bZ)\cdot\dot\bZ=-\frac{1}{N}\sum_i |\dot{ z}_i|^2.
\end{equation}

Without
 loss of generality,  we assume that an $N$-point  configuration parametrization $\bZ=(z_1,z_2,\ldots,z_N)$ is ordered as
\begin{equation}\label{sisort}
 z_1<\cdots< z_N<z_1+1,\end{equation}
and observe that if the initial data (at $t=0$) satisfies \eqref{eq_si},  then \eqref{sisort} holds  for all time due to the singularity of the interaction potential $W$ at 0.  Consistent with the periodicity of $\bx$, we    extend $z_i$ to all $i\in\mathbb{Z}$ by setting $z_{i+N}=z_i+1$ so that $\bx(z_{i+N})=\bx(z_i)$.

The determination of optimal $N$-point configurations confined to a curve or
more generally a manifold, whose pairwise interactions are governed by the  Riesz $s$-potential $W$ in \eqref{Rie} is sometimes referred to when the manifold is the unit sphere $S^d\subset \mathbb{R}^d$ and $s>0$, as the ``generalized
Thomson problem." Determining the minimal energy positions for such points
explicitly is a notoriously difficult problem for which only some very
special cases are known, even for ``small" values of $N$ (see \cite{ConSl}, \cite{BHS2019}).  One of these cases is that of the unit circle in $\mathbb{R}^2$, for which a simple convexity
argument shows that $N$ distinct equally spaced points ($N$-th roots of unity) are the unique (up to rotation) $N$-point configurations that minimize the energy for all $s>0$ and all $N\geq 2$. There are, however, several well-known theorems
that deal with the asymptotics as $N\to\infty$ for optimal configurations on
manifolds in Euclidean space. For curves in $\mathbb{R}^d$ in the hyper-singular 
case $s>1$, the following theorem was proved by Martinez-Finkelstein et. al. in \cite{MMRS}.
\begin{theorem} \label{littlepoppy}
If $s>1$ and $\Gamma$ is a rectifiable Jordan arc or closed curve embedded in $\mathbb{R}^d$ of length one with arc length parametrization ${\bf x}(s)$,
then
$$ 
\lim_{N\to \infty} \min E({\bf Z})=\zeta(s)/s,
$$ 
where the minimum is taken over all $N$-point configurations $\{{\bf x}(z_i)\}_{i=1}^N$ on $\Gamma$ and $\zeta(s)$ is the classical Riemann
zeta function. Moreover, $N$-point minimizing configurations $\{{\bf x}(z_i^*)\}_{i=1}^N$ are asymptotically uniformly distributed with respect
to arc length and, with $d_i^*:=z_{i+1}^*-z_i^*$, satisfy
\begin{equation}\label{disgen}
\sum_{i=1}^N \left| d_i^* - \frac{1}{N}\right | \to 0\,\,\,as\,\,\,N\to\infty.
\end{equation}
\end{theorem}

This theorem together with its refinement \cite{Bor2012}, which is one of the main motivations for the present work, is a special case 
of the so-called \emph{Poppy-seed bagel theorem} (see  \cite{BHS2019}) which applies to general $d$-rectifiable manifolds embedded in $\mathbb{R}^p,\, d \leq p.$

As stated in Theorem~\ref{littlepoppy},  any minimizer of the energy $E$ defined in \eqref{EDef} has to be almost uniformly distributed. This paper studies the large time behavior of \eqref{eq_si}; namely, whether $\{z_i(t)\}_{i\in\mathbb{Z}}$ are ``close to equally spaced'' as $t\rightarrow\infty$.

\section{Main results}

We will use the following quantities depending on $s$:
\begin{equation}
\cs  := \sum_{i=1}^\infty i^{-s},\quad \tcs  := \frac{\cs }{s}.
\end{equation}

Every constant $C$ or $c$ appearing in this paper   depends only on $s$ and the curve $\bx( z)$, if not stated otherwise.

\subsection{Statement of main results}

Our first main result is the following.
\begin{theorem}\label{thm1}
Let $\bx(z)$ be a non-self-intersecting $C^4$ closed curve, and let $s>1$.  For any $\epsilon>0$, there exists  $N_0$, depending on $\epsilon,s$ and the curve $\bx( z)$, such that the following holds for $N>N_0$: for the solution to \eqref{eq_si} with initial data satisfying \eqref{sisort}, there exists a positive constant $C$ such that 
\begin{equation}
E(t) \le \tcs (1+\epsilon) ,\quad \forall t \ge \frac{C}{\epsilon}.
\end{equation}
\end{theorem}

This theorem quantifies the convergence rate of the solution to \eqref{eq_si} to an almost minimal energy state. In fact, since Lemma \ref{lem_E} shows that the global minimum of $E$ is at least $\tcs (1-\epsilon)$, Theorem \ref{thm1} shows that, after time $\calO(1/\epsilon)$, the energy will decay to the global minimum up to an error of $\calO(\epsilon)$. This can be viewed as an energy decay rate of $\calO(1/t)$ being \emph{independent of} the number of particles $N$, as long as $N$ is large enough. 

Our second main result shows that upper bounds on the energy of  $N$-point configurations  such as provided by Theorem~\ref{thm1}  impose  geometrical constraints on the distribution of these configurations showing that they are near optimal configurations.
 \begin{theorem}\label{thm2}
For given $\epsilon>0$ and $s>1$, there is some $N_0$   depending on $s$ and $\epsilon$  such that if  $N>N_0$ and  $\bZ=\{z_i\}_{i=1}^N$ satisfies
\begin{equation}
E(\bZ) \le \tcs (1+\epsilon),
\end{equation}
then the mean absolute deviation of  $d_i=z_{i+1}-z_i$, $i=1,2,\ldots, N$, satisfies
\begin{equation}\label{Zmad1}
\frac{1}{N}\sum_{i=1}^N \left |d_i-\frac{1}{N}\right |\le 2\left(\frac{2\tcs}{s+1}\right)^{1/2}\frac{\epsilon^{1/2}}{N},
\end{equation}
and for all $a\in \mathbb{R}$ and  $0<L<1$, we have
\begin{equation}\label{disc1}
  \left|\frac{\#\{i:\,[z_i,z_{i+1})\subset [a,a+L)\}}{N} - L\right|\le \left[L(1-L)\tcs\right]^{1/2}(2\epsilon)^{1/2}.\end{equation}
Consequently, under the assumptions of Theorem~\ref{thm1}, the conclusions \eqref{Zmad1} and \eqref{disc1} hold for $N$ sufficiently large and $t\ge C/\epsilon$.
\end{theorem}

The proof of Theorem~\ref{thm1} is given in 
Sections 3-6. Below we discuss the motivation
for the argument used in its proof.  The proof of Theorem~\ref{thm2} is given in Section 7.

\subsection{Outline of the proof of Theorem~\ref{thm1}}

It is known that the global minimizer of $E$ defined in \eqref{EDef} converges to the uniform distribution as $N\rightarrow\infty$; therefore it is natural to expect that, for large $N$, the gradient flow \eqref{eq_si}  converges to some limiting configuration which is nearly equally distributed. However, we encounter the following difficulties:
\begin{itemize}
\item When the curve $\bx( z)$ is not convex, the energy $E$ is not necessarily a convex function of $\{ z_i\}$.
\item The global minimizer of $E$ may not be unique, and there may be local minimizers and saddle points.
\end{itemize}

To handle these difficulties, we manage to extract some ideas from the mean field limit of \eqref{eq_si}. In fact, it is proved in \cite{OE} that the analog of \eqref{eq_si} on the real line has the porous medium equation
\begin{equation}\label{porous}
\partial_t \rho = \cs\partial_{zz}(\rho^{s+1})
\end{equation}
as its mean field limit, under certain assumptions on the initial data. This mean field limit can be understood intuitively as follows:
\begin{itemize}
\item Due to the fast decay of $W(\bx)$ for large $|\bx|$, the particle interaction is \emph{localized} when $N$ is large, meaning that typically the interaction between particles with large distances can be neglected, at least for a fixed time interval $[0,T]$. The same holds for the curvature effect, i.e., the difference between  \eqref{eq_si} and its analog on the real line.
\item Due to the strong localized repulsion, particles tend to distribute \emph{locally} in a uniform way, similar to the local equilibrium in kinetic theory. This means, in a short interval $I$ of length $\delta$ (which is still long enough to contain a large number of particles), the particles are approximately uniformly distributed. However, the particle density may still have variation on a macroscopic scale, according to some density profile $\rho(t,z)$.
\item In a short interval $I$ of length $\delta$, if the particles inside are uniformly distributed with density $\rho$ (i.e., the distance between adjacent particles is approximately $1/(N\rho)$, and the total number of particles inside is approximately $\delta N \rho$), then the total energy of the particles inside is approximately
\begin{equation}
N^{-s-1}\sum_{z_i\in I} \sum_{j\ne i} \frac{|z_i-z_j|^{-s}}{s} \approx N^{-s-1} (\delta N\rho)\cdot \sum_{j\in\mathbb{Z},j\ne 0} \frac{|j/(N\rho)|^{-s}}{s} = 2\tcs \rho^{s+1}\delta.
\end{equation}
Summing   all the short intervals (and symmetrizing in $i$ and $j$), this gives a Riemann sum which approximates
\begin{equation}
E(\bZ) \approx \tcs \int \rho^{s+1}\rd{z}.
\end{equation}
Then notice that \eqref{eq_si} is the gradient flow of $E$, while \eqref{porous} is exactly the Wasserstein-2 gradient flow of the above right-hand side [RHS].
\end{itemize}

Although  mean field limits are generally not true on the whole time axis $[0,\infty)$, we can indeed get some ideas from the energy structure of \eqref{porous}. To motivate the proof of Theorem \ref{thm1}, we start from the following two properties of the porous medium equation \eqref{porous}:
\begin{itemize}
\item Suppose at time $t$, there are two points $z_M$ and $z_S$ such that $\rho(t,z_M)>\rho(t,z_S)$ (assuming $z_M<z_S$ without loss of generality). Then
\begin{equation}\label{porous1}
\int_{z_M}^{z_S} \Big(-\frac{s+1}{s}\cs\partial_z(\rho^s)\Big)\cdot \rho(t,z)\rd{z} = \cs(\rho(t,z_M)^{s+1}-\rho(t,z_S)^{s+1})> 0,
\end{equation}
where the term $-\frac{s+1}{s}\cs\partial_z(\rho^s)$ is the transport velocity of the porous medium equation, by writing $\partial_{zz}(\rho^{s+1}) = \frac{s+1}{s}\partial_z (\rho\partial_z(\rho^s))$. This means that we have a lower bound on the energy dissipation rate:
\begin{equation}
\begin{split}
\frac{\rd}{\rd{t}} \int\rho^{s+1}\rd{z} & = -\frac{s+1}{s}\cs\int |\partial_z(\rho^s)|^2\rho\rd{z} \\
 & \le -\frac{s+1}{s}\cs\cdot\frac{\left(\int (-\partial_z(\rho^s))\rho\rd{z}\right)^2}{\int \rho\rd{z}}.
 \end{split}
\end{equation}
Since the total amount of energy is finite, $|\rho(t,z_M)-\rho(t,z_S)|$ will eventually get small after a long time. In particular, for some large $T$, $\sup_z \rho(T,z)$ will get close to the average density $\int \rho\rd{z}/\int \rd{z}$.

\item The porous medium equation \eqref{porous} obeys the maximum principle:
\begin{equation}
\sup_z \rho(t,z) \text{ is decreasing in }t.
\end{equation}
This means that, once $\sup_z \rho(T,z)$ gets close to the average density, it cannot become large again, which means $\rho(t,z)$ will be close to a uniform distribution for all $t\ge T$.
\end{itemize}

To prove Theorem \ref{thm1}, we aim to find the analogues of the above two properties for \eqref{eq_si}:
\begin{itemize}
\item In the case of a flat $\mathbb{T}$, we prove Lemma \ref{lem_M2} as the counterpart of the first property. It says, once we have an interval in which the `density' (number of particles divided by interval length)  is small, then we can find a place to cut the interval, such that the total repulsion force between left and right is small. This concept of `total repulsion force' is the counterpart of the term $\rho(t,z_S)^{s+1} $ in \eqref{porous1}.

\item We establish Lemma \ref{lem_close} as the counterpart of the second property. It says that the distance $\delta$ between the closest pair of particles basically cannot decrease (see \eqref{lem_close0_1}, whose RHS is $o(1)$), in correspondence to the decreasing property. Furthermore, for reasonable situations, we have the lower bound \eqref{lem_close_2} for the `total repulsion force' at this closest pair of particles, serving as the counterpart of the term $\rho(t,z_M)^{s+1} $ in \eqref{porous1}.
\end{itemize}

Finally, we have to deal with the finite-$N$ effect and the curvature effect from $\bx(z)$, which may produce errors to the above two properties. Therefore, we need to keep track of the $N$-dependence of error terms, as well as using the smoothness of curve $\bx(z)$, to show that all such error terms are small enough.

\section{Lemmas on total repulsion cut}

For a given set of points $x_0<\dots<x_N \in\mathbb{R}$, we define the {\it total repulsion} of the {\it cut} at $x_k,x_{k+1}$ by
\begin{equation}
P_k = P_k(x_0,\dots,x_N) := \sum_{i,j:\, 0\le i\le k<j\le N} (x_j-x_i)^{-s-1}
\end{equation}

The main purpose of this section is to prove the following lemma:
\begin{lemma}\label{lem_M2}
For any $0<\epsilon\le 0.01$, there exists $N_0=N_0(\epsilon)$  such that if $N>N_0$, then for any $0= x_0 < \cdots < x_N = 1$  there exists an index $i_S$ such that
$(x_{i_S},x_{i_S+1})\bigcap (\epsilon_1,1-\epsilon_1) \ne \emptyset$ with $\epsilon_1 = \frac{\epsilon}{3(1+s)}$, and
\begin{equation}\label{lem_M2_1}
P_{i_S} \le (1+\epsilon) \cs  N^{s+1}.
\end{equation}
\end{lemma}

Notice that the total repulsion between two infinite sets of equally distributed points $\{\frac{i}{N}\}_{i=0}^\infty$ and $\{-\frac{j}{N}\}_{j=1}^\infty$ is
\begin{equation}
\sum_{i=0}^\infty \sum_{j=1}^\infty \Big(\frac{i+j}{N}\Big)^{-s-1} = N^{s+1} \sum_{i=1}^\infty i\cdot i^{-s-1} = \cs N^{s+1}.
\end{equation}
Therefore, Lemma \ref{lem_M2} tells us that one can find an index $i_S$  such that the total repulsion for $k=i_S$ there is at most slightly more than for equally distributed points.

The proof of this lemma follows a min-max type argument. Let $0\le i_L < i_R \le N$ be two indices. Define
\begin{equation}\label{F}\begin{split}
& F_m(x_{i_L+1},\dots,x_{i_R-1}) := \min_{i_L\le k \le i_R-1}P_k, \\
\end{split}\end{equation}
viewing those $x_i$'s with $i\le i_L$ or $i\ge i_R$ as fixed. $F_m$ are defined on
\begin{equation}
\begin{split}
\mathbb{R}&^{i_R-i_L-1}_{\textnormal{sort}} (x_{i_L},x_{i_R}) \\
 & = \{(x_{i_L+1},\dots,x_{i_R-1})\in\mathbb{R}^{i_R-i_L-1}: x_{i_L}< x_{i_L+1} < \cdots < x_{i_R-1}<x_{i_R}\},
 \end{split}
\end{equation}
which is a convex open set.

In the following lemma we describe the global maximum of $F_m$ as a function of $x_{i_L+1},\dots,x_{i_R-1}$.
\begin{lemma}\label{lem_F}
The global maximum of $F_m$ on $\mathbb{R}^{i_R-i_L-1}_{\textnormal{sort}}(x_{i_L},x_{i_R})$ is achieved at the same point $X^*= (x_{i_L+1}^*,\dots,x_{i_R-1}^*)$, which is the only point satisfying
\begin{equation}\label{lem_F_1}
P_{i_L} = \cdots = P_{i_R-1}.
\end{equation}
Furthermore, $X^*$ is the unique global minimizer of the energy functional
\begin{equation}
\cE(x_{i_L+1},\dots,x_{i_R-1}) := \sum_{i,j:\, 0\le i<j\le N} (x_j-x_i)^{-s},
\end{equation}
and 
\begin{equation}\label{lem_F_2}
F_m(X^*)  = \frac{1}{x_{i_R}-x_{i_L}}  \sum_{0\le i < j \le N,\, i< i_R ,\,j> i_L} (x_{\min\{j,i_R\}}^*-x_{\max\{i,i_L\}}^*) (x_j^*-x_i^*)^{-s-1},
\end{equation}
with $x_i^*:=x_i$ for $0\le i \le i_L$ or $i_R\le i \le N$.
\end{lemma}

Notice that the RHS of \eqref{lem_F_2} is exactly $\cE(X^*)$ if $i_L=0,\,i_R=N$.

\begin{proof}

{\bf STEP 1}: Show that the global maximum of $F_m$ is achieved inside $\mathbb{R}^{i_R-i_L-1}_{\textnormal{sort}}(x_{i_L},x_{i_R})$.

In fact, one can extend the definition of $F_m$ to the closure of $\mathbb{R}^{i_R-i_L-1}_{\textnormal{sort}}(x_{i_L},x_{i_R})$ by interpreting $(x_j-x_i)^{-s-1}$ as infinity when $x_j=x_i$, and $F_m$ remains continuous. We show that the (global) maximum of $F_m$ on the closure of $\mathbb{R}^{i_R-i_L-1}_{\textnormal{sort}}(x_{i_L},x_{i_R})$ is not achieved at boundary. In fact, at any boundary point, one has either $x_{k_1-1}<x_{k_1}=x_{k_1+1}=\cdots=x_{k_2}<x_{k_2+1}$ for some $i_L< k_1<k_2 < i_R-1$, or $x_{i_L}=x_{i_L+1}$, or $x_{i_R}=x_{i_R-1}$. We show that maximum is not achieved in the first case, and the other cases can be handled similarly.

In the first case, by replacing $x_{k_1}$ and $x_{k_2}$ by $x_{k_1}-\delta$ and $x_{k_2}+\delta$ respectively, with $\delta>0$ small enough, we claim that $F_m$ is decreased. First of all, $P_k$ with $k_1\le k <k_2$ is much larger than $F_m$ if $\delta$ is small, and thus the minimum in \eqref{F} is achieved elsewhere. For any $j$ with $k_2 < j \le i_R$,
\begin{equation}\label{dd}\begin{split}
 \frac{\rd}{\rd{\delta}}\Big|_{\delta=0}&[(x_j-(x_{k_1}-\delta))^{-s-1} + (x_j-(x_{k_2}+\delta))^{-s-1}] \\
= & (-s-1)[(x_j-(x_{k_1}-\delta))^{-s-2} - (x_j-(x_{k_2}+\delta))^{-s-2}]|_{\delta=0} > 0,
\end{split}\end{equation}
since $-s-1<0$ and $x_j-x_{k_1} > x_j-x_{k_2}$. Similarly for any $j$ with $i_L\le j < k_1$,
\begin{equation}
\frac{\rd}{\rd{\delta}}\Big|_{\delta=0}[((x_{k_1}-\delta)-x_j)^{-s-1} + ((x_{k_2}+\delta)-x_j)^{-s-1}]  > 0.
\end{equation}
This shows that for any $k$ with $k_2 \le k \le i_R-1$ or $i_L\le k < k_1$, $P_k$ is increased if $\delta>0$ is small. Thus $F_m$ is increased. By doing this $[(k_2-k_1)/2]$ times, one reaches the interior of $\mathbb{R}^{i_R-i_L-1}_{\textnormal{sort}}(x_{i_L},x_{i_R})$ while making $F_m$ increased.

{\bf STEP 2}: Show \eqref{lem_F_1} for $X^m$, the global maximum of $F_m$.

From STEP 1, the maximum of $F_m$ is achieved in the interior of $\mathbb{R}^{i_R-i_L-1}_{\textnormal{sort}}(x_{i_L},x_{i_R})$, say at $X^m=(x_{i_L+1}^m,\dots,x_{i_R-1}^m)$. Suppose on the contrary that \eqref{lem_F_1} is not true, then there exists $k$ with $i_L\le k \le i_R-1$ such that $P_k>F_m$. If $i_L< k < i_R-1$, then by replacing $x_{k}$ and $x_{k+1}$ by $x_{k}-\delta$ and $x_{k+1}+\delta$ respectively, with $\delta>0$ small enough, we can show similarly (see \eqref{dd}) that $P_k$ is slightly decreased, while still being larger than $F_m$, and all other $P_{k'},\,k'\ne k$, are increased. Thus $F_m$ is increased, which is a contradiction against the maximality.  If $k=i_L$ or $k=i_R-1$, then adjusting $x_k$ or $x_{k+1}$ respectively in a similar way will give the same conclusion.


{\bf STEP 3}: Show that \eqref{lem_F_1} is exactly the characterizing condition of the unique global minimizer of $\cE$.

Since $\cE$ is convex and going to infinity near the boundary, the global minimizer of $\cE$ on $\mathbb{R}^{i_R-i_L-1}_{\textnormal{sort}}(x_{i_L},x_{i_R})$ is clearly unique, calling it $X^*$, characterized by
\begin{equation}\label{lem_F_1e}
\begin{split}
\partial_k \cE & = -s\cdot \Big(\sum_{i:\, 0\le i < k} (x_k-x_i)^{-s-1} \\
 & \qquad \qquad - \sum_{i:\,k<i \le N} (x_i-x_k)^{-s-1} \Big) = 0,\quad \forall i_L+1\le k \le i_R-1.
 \end{split}
\end{equation}
Notice that the quantity in the above parenthesis is exactly $P_k-P_{k-1}$. Therefore \eqref{lem_F_1e} is equivalent to \eqref{lem_F_1}. Since $X^*$ is the unique point satisfying \eqref{lem_F_1e}, and $X^m$ satisfies \eqref{lem_F_1}, these two points coincide.

{\bf STEP 4}: Show \eqref{lem_F_2}.

Notice that
\[
\begin{split}
\sum_{k=i_L}^{i_R-1} (x_{k+1}-x_k)P_k = &  \sum_{k=i_L}^{i_R-1}\sum_{i,j:\, 0\le i\le k<j\le N} (x_{k+1}-x_k) (x_j-x_i)^{-s-1} \\
= & \sum_{0\le i < j \le N} \sum_{k=\max\{i,i_L\}}^{\min\{j,i_R\}-1} (x_{k+1}-x_k) (x_j-x_i)^{-s-1} \\
= & \sum_{0\le i < j \le N,\, i< i_R ,\,j> i_L} (x_{\min\{j,i_R\}}-x_{\max\{i,i_L\}}) (x_j-x_i)^{-s-1}. \\
\end{split}
\]
At $X^*$, we have $F_m = P_k,\,i_L\le k \le i_R-1$. Thus \eqref{lem_F_2} follows.

\end{proof}

\begin{proof}[Proof of Lemma \ref{lem_M2}]
We apply Lemma \ref{lem_F} with
\begin{equation}
i_L = \max\{i: x_i < \epsilon_1\},\quad i_R = \min\{i: x_i > 1- \epsilon_1\}.
\end{equation}
Then we get
\begin{equation}\begin{split}
F_m(X) &\le  F_m(X^*) \\
  &= \frac{1}{x_{i_R}-x_{i_L}}  \sum_{0\le i < j \le N,\, i< i_R ,\,j> i_L} (x_{\min\{j,i_R\}}^*-x_{\max\{i,i_L\}}^*) (x_j^*-x_i^*)^{-s-1}  \\
  & \le   \frac{1}{x_{i_R}-x_{i_L}}\sum_{0\le i < j \le N,\, i< i_R ,\,j> i_L} (x_j^*-x_i^*)^{-s} \\
 & \le  \frac{1}{1-2\epsilon_1}\sum_{0\le i < j \le N,\, i< i_R ,\,j> i_L} (x_j^*-x_i^*)^{-s}
\end{split}\end{equation}
for $X=(x_{i_L+1},\dots,x_{i_R-1})$. Notice that
\begin{equation}
\begin{split}
\sum_{0\le i < j \le N,\, i< i_R ,\,j> i_L} (x_j-x_i)^{-s} &= \cE(x_{i_L+1},\dots,x_{i_R-1})-C_0,\\
 & \quad C_0:= \sum_{i_R\le i < j \le N \text{ or } 0\le i < j \le i_L}(x_j-x_i)^{-s}
\end{split}
\end{equation}
for any $X=(x_{i_L+1},\dots,x_{i_R-1})$, where $C_0$ is independent of $X$. Therefore
\begin{equation}
F_m(X) \le \frac{1}{1-2\epsilon_1}(\cE(X^*)-C_0).
\end{equation}

To bound $\cE(X^*)$ from above, we construct
\begin{equation}
\tilde{x}_i = \epsilon_1 + (1-2\epsilon_1)\frac{i}{N},\quad i=0,\dots,N,
\end{equation}
and denote
\begin{equation}
\tilde{\tilde{x}}_i = \left\{\begin{split}& \tilde{x}_i,\quad i_L+1\le i \le i_R-1, \\ & x_i,\quad \text{elsewhere.}\end{split}\right.
\end{equation}
Then by the minimality of $\cE(X^*)$,
\begin{equation}\begin{split}
\cE(X^*) \le & \cE(\tilde{x}_{i_L+1},\dots,\tilde{x}_{i_R-1}) \\
= & C_0 + \sum_{0\le i < j \le N,\, i< i_R ,\,j> i_L} (\tilde{\tilde{x}}_j-\tilde{\tilde{x}}_i)^{-s} \\
\le  & C_0 + \sum_{0\le i < j \le N,\, i< i_R ,\,j> i_L} (\tilde{x}_j-\tilde{x}_i)^{-s} \\
\le  & C_0 + (N+1)\sum_{i=1}^\infty \Big((1-2\epsilon_1)\frac{i}{N}\Big)^{-s}  \\
= & C_0 + (1-2\epsilon_1)^{-s} \cs (N+1)N^{s},
\end{split}\end{equation}
where the second inequality is because when changing from $\tilde{\tilde{x}}$ to $\tilde{x}$, we have
\begin{equation}
\tilde{\tilde{x}}_j-\tilde{\tilde{x}}_i =\begin{cases}\tilde{x}_j-\tilde{x}_i, &  i_L+1 \le i < j \le i_R-1;\\
\tilde{x}_j-x_i \ge \tilde{x}_j-\epsilon_1 \ge \tilde{x}_j-\tilde{x}_i & i \le i_L <  j \le i_R-1;\\
x_j-\tilde{x}_i \ge (1-\epsilon_1) - \tilde{x}_i \ge \tilde{x}_j-\tilde{x}_i&  i_L +1\le i  < i_R \le j;\\
x_j-x_i \ge (1-\epsilon_1) - \epsilon_1 \ge \tilde{x}_j-\tilde{x}_i & i\le i_L  < i_R \le j;
\end{cases}\end{equation}
which includes all the cases appearing in the summation. Therefore we finish the proof by
\begin{equation}
\begin{split}
F_m(X) & \le (1+\frac{1}{N})(1-2\epsilon_1)^{-s-1} \cs N^{s+1} \\
 & \le (1+\frac{1}{N})(1+2.5(s+1)\epsilon_1) \cs N^{s+1} \le (1+\epsilon)\cs N^{s+1}
\end{split}
\end{equation}
for $\epsilon_1\le \frac{0.01}{3(s+1)}$ and $N$ large enough, where the second inequality uses
\begin{equation}
(1-2\epsilon_1)^{-s-1} \le (1+2.2\epsilon_1)^{s+1} \le e^{2.2\epsilon_1(s+1)} \le 1+2.5(s+1)\epsilon_1.
\end{equation}

\end{proof}

\begin{remark}
Under the same assumptions as in Lemma \ref{lem_M2}, one can show the existence of an index $i_M$ such that $P_{i_M}\ge (1-\epsilon)\cs N^{s+1}$. We omit the details for this result because it will not be used in the proof of Theorem \ref{thm1}.
\end{remark}

\section{Approximation by flat torus}

%

For given $ z_1(t),\dots, z_N(t)$ satisfying \eqref{sisort}, define the closest pairwise distance and the `maximal density', respectively, by
\begin{equation}\label{delta}
\delta(t) := \min_{1\le i \le N}( z_{i+1}(t)-  z_i(t)),\quad \rho_M(t) := \frac{1}{N\delta(t)}
\end{equation}
with $z_{N+1}$ understood as $z_1$.
Furthermore, at a fixed time $t$, we set
\begin{equation}
i_M  := \text{argmin}_i ( z_{i+1}-  z_i)
\end{equation}
as the index of the closest pair of particles.
Finally, we define
\begin{equation}
d(y,z) := \min_{k\in\mathbb{Z}} |y-z+k|
\end{equation}
as the distance between $y$ and $z$ on the flat torus. It is clear that $d(y,z)=|y-z|$ if $|y-z|\le \frac{1}{2}$.

\begin{lemma}\label{lem_nsi}
There exists $r_0>0$ such that
\begin{equation}\label{lem_nsi_1}
 |\bx(y)-\bx(z)| \ge \min\{\frac{1}{2}d(y,z),r_0\},\quad \forall y,z.
\end{equation}
\end{lemma}

See Figure \ref{fig1} for an illustration of \eqref{lem_nsi_1}.

\begin{figure}
\begin{center}
  \includegraphics[width=.6\linewidth]{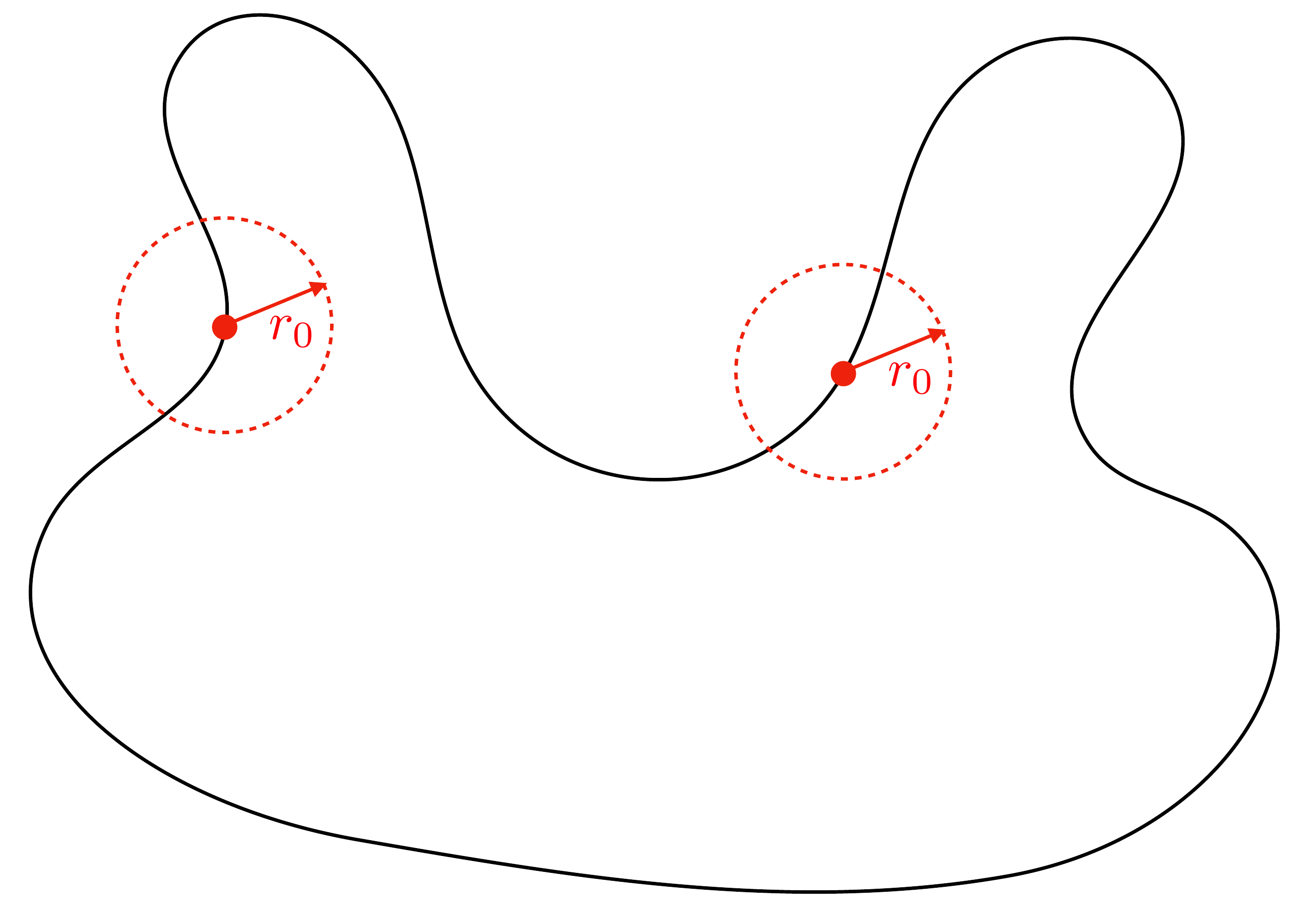}
  \caption{The number $r_0$ in Lemma \ref{lem_nsi} is the range for which $\bx(z)$ can be approximated by a local Taylor expansion near $\bx(y)$  for any fixed $y$. }
\label{fig1}
\end{center}
\end{figure}

\begin{proof}

First, by the Taylor expansion
\begin{equation}
\bx( y)-\bx( z) = ( y- z)\bx'( y) + \calO(( y- z)^2)
\end{equation}
we see that
\begin{equation}\label{nsi1}
\frac{1}{2}|y-z| \le |\bx( y)-\bx( z)| \le \frac{3}{2}|y-z|
\end{equation}
if $|y-z|\le r_1$ is small enough.

Consider the continuous function
\begin{equation}
F(y,z) = |\bx(y)-\bx(z)|
\end{equation}
defined on $\{(y,z)\in \mathbb{T}^2:d(y,z) \ge r_1\}$ which is compact. Since $\bx(z)$ is non-self-intersecting, $F$ is everywhere positive, and achieves its positive minimum on this set, calling it $r_0$.

To show \eqref{lem_nsi_1}, if $d(y,z) \ge r_1$, then the definition of $r_0$ gives
\begin{equation}
|\bx( y)-\bx( z)| \ge r_0.
\end{equation}
If $ d(y,z) = |y-z| < r_1$, then \eqref{nsi1} gives
\begin{equation}
|\bx( y)-\bx( z)| \ge \frac{1}{2}|y-z| = \frac{1}{2}d(y,z).
\end{equation}

\end{proof}

\begin{lemma}\label{lem_appr}
There exist $C_R>0$  and $r_0>0$, depending on the curve $\bx( z)$ and $s$, such that for any $ y\ne  z\in\mathbb{T}$ with $d(y,z)\le r_0$, we have 
\begin{equation}\label{lem_appr_1}
|\nabla W(\bx( y)-\bx( z))\cdot \bx'( y) -  W'( y- z)(1+\kappa(y)|y-z|^2)| \le C_R| y- z|^{-s+2},
\end{equation}
where
\begin{equation}
\kappa(z):= \frac{s-2}{24}|\bx''(z)|^2.
\end{equation}
Furthermore,
\begin{equation}\label{lem_appr_2}
\sgn(\nabla W(\bx( y)-\bx( z))\cdot \bx'( y)) = \sgn(W'( y- z)).
\end{equation}
If  $y,z$ and $\tilde{y}$ additionally satisfy $\tilde{y}-1<z<y<\tilde{y}$, then
\begin{equation}\label{lem_appr_11}\begin{split}
& \Big|\Big(\nabla W(\bx( y)-\bx( z))\cdot \bx'( y) -  W'( y- z)(1+\kappa(y)|y-z|^2)\Big)\\
& -\Big(\nabla W(\bx(\tilde{y})-\bx( z))\cdot \bx'(\tilde{y}) -  W'(\tilde{y}- z)(1+\kappa(y)|\tilde{y}-z|^2)\Big)\Big| \\
\le & C_R\min\{d(y,z),d(\tilde{y},z)\}^{-s+1}\cdot |y-\tilde{y}|
\end{split}\end{equation}
and the same inequality holds if $\kappa(y)$ is replaced by $\kappa(\tilde{y})$.

Moreover, for any $r_1>0$, there exists $C_0(r_1)>0$ such that
\begin{equation}\label{lem_appr_3}
|\nabla W(\bx( y)-\bx( z))| \le C_0(r_1),\quad \forall d(y,z)> r_1.
\end{equation}
\end{lemma}


\begin{proof}

We assume hereafter that  $r_0$ is sufficiently small so that Lemma \ref{lem_nsi} applies.

{\bf STEP 1}: We first prove  \eqref{lem_appr_1} and \eqref{lem_appr_2} with the assumption $d(y,z)=|y-z|\le r_0$.

By Taylor expansion for $| y- z|$ small,
\begin{equation}
\bx( y)-\bx( z) = ( y- z)\bx'( y) - \frac{( y- z)^2}{2}\bx''( y) + \frac{( y- z)^3}{6}\bx'''( y) + \calO(( y- z)^4)
\end{equation}
where the error term involves $\|\bx^{(4)}\|_{L^\infty}$. Since the curve length parametrization satisfies $|\bx'( z)|=1$, one obtains
\begin{equation}\label{xs1xs20}
\bx''( z)\cdot\bx'( z)=0,\quad \bx'''( z)\cdot\bx'( z) + |\bx''(z)|^2 = 0
\end{equation}
by differentiating with respect to $z$. Then we have
\begin{equation}\label{xs1xs2}\begin{split}
|\bx( y)-\bx( z)|^2 = & ( y- z)^2\Big[1 + ( y- z)^2\Big(\frac{1}{3}\bx'( y)\cdot\bx'''( y) + \frac{1}{4}|\bx''( y)|^2\Big) \\
 & \qquad \qquad  + \calO(( y- z)^3)\Big] \\
= & ( y- z)^2\left[1 - ( y- z)^2\frac{1}{12}|\bx''( y)|^2 + \calO(( y- z)^3)\right], \\
\end{split}\end{equation}
and
\begin{equation}\label{xs1xs21}\begin{split}
(\bx( y)-\bx( z))\cdot \bx'( y) = & ( y- z)\left[1 + ( y- z)^2\frac{1}{6}\bx'( y)\cdot\bx'''( y) + \calO(( y- z)^3)\right] \\
= & ( y- z)\left[1 - ( y- z)^2\frac{1}{6}|\bx''( y)|^2 + \calO(( y- z)^3)\right]. \\
\end{split}\end{equation}

Also, when $r_0$ is small, we have $O((y-z)^2) \le 1/2$, and thus \eqref{xs1xs2} implies
\begin{equation}\label{xs1xs22}
\begin{split}
|\bx( y)&-\bx( z)|^{-s-2} \\
 &= | y- z|^{-s-2}\left[1 -( y- z)^2\frac{-s-2}{2}\cdot \frac{1}{12}|\bx''( y)|^2 + \calO(( y- z)^3)\right].
\end{split}
\end{equation}
Multiplying this with \eqref{xs1xs21} gives
\[
\begin{split}
 \nabla W(\bx( y)-\bx( z))&\cdot \bx'( y) \\
= & |\bx( y)-\bx( z)|^{-s-2}(\bx( y)-\bx( z))\cdot \bx'( y) \\
= & | y- z|^{-s-2}( y- z)\left[1 + ( y- z)^2\frac{s-2}{24}|\bx''( y)|^2 + \calO(( y- z)^3)\right]
\end{split}
\]
and \eqref{lem_appr_1} with $ | y- z|\le r_0$ follows. Then \eqref{lem_appr_2} follows from the fact that $|W'( y- z)(1+\kappa(y)|y-z|^2)| \ge |y-z|^{-s-1}/2 \ge C_R|y-z|^{-s+2}$ when $|y-z|$ is small enough.

{\bf STEP 2}:  Here we prove \eqref{lem_appr_3}.

If $| y- z|>r_1$, then by Lemma \ref{lem_nsi}, there exists constant $r_1'= \min\{r_1/2,r_0/2\}>0$ such that
\begin{equation}
|\bx( y)-\bx( z)| \ge r_1'.
\end{equation}
Then it follows that
\begin{equation}
|\nabla W(\bx( y)-\bx( z))| = |\bx( y)-\bx( z)|^{-s-1}  \le (r_1')^{-s-1}=:C_0(r_1).
\end{equation}
This gives  \eqref{lem_appr_3}. 

{\bf STEP 3}: Finally we prove \eqref{lem_appr_11}.


We define a function\footnote{As auxiliary functions, $\phi$ may refer to different functions in different proofs.}
\begin{equation}
\phi(z) = \nabla W(\bx( z)-\bx( z))\cdot \bx'( z) - W'(z-z)(1+\kappa(y)|z-z|^2)
\end{equation}
and then the LHS of \eqref{lem_appr_11} is $|\phi(y)-\phi(\tilde{y})| = |\phi'(\xi)|\cdot |\tilde{y}-y|$ for some $\xi\in (y,\tilde{y})$.

Write $\xi = y + \alpha (\tilde{y}-y),\,0\le \alpha \le 1$. By assumption, $d(y,z) = y-z \le r_0$ is small. Therefore
%
\begin{equation}
|\xi-z| = |y-z| + \alpha|\tilde{y}-y| \in [\,|y-z|,2r_0]
\end{equation}
since both $y-z$ and $\tilde{y}-y$ are positive.

Then we compute
\begin{equation}\begin{split}
\phi'(\xi) = & \bx'(\xi)^T\cdot \nabla^2 W(\bx(\xi)-\bx( z))\cdot \bx'(\xi) + \nabla W(\bx( \xi)-\bx( z))\cdot \bx''(\xi) \\
& - W''(\xi-z)(1+\kappa(y)|\xi-z|^2) - W'(\xi-z)\kappa(y)\cdot 2(\xi-z)
\end{split}\end{equation}
where
\begin{equation}
\nabla^2 W(\bar{\bx}) = -|\bar{\bx}|^{-s-2}I + (s+2)|\bar{\bx}|^{-s-4}\bar{\bx}\bar{\bx}^T,\quad \bar{\bx} := \bx(\xi)-\bx( z).
\end{equation}
Therefore, using $|\bx'(\xi)|=1$,
\begin{equation}\begin{split}
\phi'(\xi) =  &-|\bar{\bx}|^{-s-2}  + (s+2)|\bar{\bx}|^{-s-4} (\bx'(\xi)\cdot \bar{\bx})^2 - |\bar{\bx}|^{-s-2}(\bx''(\xi)\cdot \bar{\bx}) \\
& - (s+1)|\xi-z|^{-s-2}(1+\kappa(y)|\xi-z|^2) \\
 & + |\xi-z|^{-s-2}(\xi-z)\kappa(y)\cdot 2(\xi-z) \\
 = & |\bar{\bx}|^{-s-2} \Big[-1 + (s+2)|\bar{\bx}|^{-2} (\bx'(\xi)\cdot \bar{\bx})^2 - (\bx''(\xi)\cdot \bar{\bx})\Big] \\
& - (s+1)|\xi-z|^{-s-2}(1+\kappa(y)|\xi-z|^2) \\
 & + |\xi-z|^{-s-2}(\xi-z)\kappa(y)\cdot 2(\xi-z) \\\end{split}\end{equation}
\eqref{xs1xs20}, \eqref{xs1xs21} and \eqref{xs1xs22} with $y$ replaced by $\xi$ (which is allowed since $|z-\xi|\le 2r_0$, by replacing $r_0$ with a smaller one if necessary), give
\[
\begin{split}
 \phi'& (\xi) \\
  & =  |\xi-z|^{-s-2}\cdot \left[1 -(\xi- z)^2\frac{-s-2}{2} \frac{1}{12}|\bx''(\xi)|^2+\calO\right] \\
& \quad \cdot \Big[ -1 + (s+2)\left(1 +(\xi- z)^2\frac{1}{12}|\bx''(\xi)|^2+\calO\right)\cdot \left(1 - (\xi- z)^2\frac{1}{6}|\bx''(\xi)|^2+\calO\right)^2 \\
 & \qquad + (\xi-z)^2\frac{1}{2}|\bx''(\xi)|^2+\calO  \Big] \\
&  \qquad - (s+1)|\xi-z|^{-s-2}(1+\kappa(y)|\xi-z|^2) + |\xi-z|^{-s-2}(\xi-z)\kappa(y)\cdot 2(\xi-z) \\
& =  |\xi-z|^{-s-2}\cdot \Big[ (s+1) \\
 & \hspace*{3cm} + (\xi-z)^2\cdot\Big( (s+1)\frac{s+2}{24} + \frac{s+2}{12} - \frac{s+2}{3} + \frac{1}{2} \Big)|\bx''(\xi)|^2  +\calO \Big] \\
& \quad - |\xi-z|^{-s-2}\Big[ (s+1) + (\xi-z)^2\cdot\Big( (s+1)\kappa(y) - 2\kappa(y)\Big) +\calO \Big] \\
& =  |\xi-z|^{-s-2}\cdot \Big[ (\xi-z)^2(s-1)\kappa(\xi) -(\xi-z)^2(s-1)\kappa(y) +\calO \Big] \\
& =  \calO(|\xi-z|^{-s+1})
\end{split}
\]
where $\calO$ refers to $\calO((\xi-z)^3)$, and in the last equality we used $|\kappa(y)-\kappa(\xi))| \le \|\kappa'\|_{L^\infty}\cdot |y-\xi|\le \|\kappa'\|_{L^\infty}\cdot |y-\tilde{y}|$. This gives \eqref{lem_appr_11}.

When replacing $\kappa(y)$ by $\kappa(\tilde{y})$, the total change on the LHS of \eqref{lem_appr_11} is no more than $\calO(|y-z|^{-s-1}\cdot |y-z|^2\cdot |y-\tilde{y}|)$ since $|\kappa(y)-\kappa(\tilde{y})| \le \|\kappa'\|_{L^\infty}\cdot |y-\tilde{y}|$, thus controled by the RHS.
\end{proof}

\begin{lemma}\label{lem_E}
For any $\epsilon>0$, there exists (large) $N_0$, depending on $\epsilon,\ s$ and the curve $\bx( z)$, such that the following holds for $N>N_0$ and any positions of the particles $ \bZ=\{z_1,\dots, z_N\}$:
\begin{equation}\label{lem_E_1}
\tcs(1-\epsilon) \le E(\bZ) \le \tcs(1+\epsilon)\rho_M^{s}
\end{equation}
\end{lemma}
\begin{proof}
We first prove the right-hand inequality of \eqref{lem_E_1}.  We rewrite \eqref{EDef}
\begin{equation}
2E(\bZ) = N^{-s-1}\sum_{i} \sum_{j\ne i}W(\bx( z_i)-\bx( z_j)).
\end{equation}
For each fixed $i$, let $i_L,\dots,i_R$ be the indices $j$ with $|z_i- z_j| \le r_0$, where $r_0>0$ is a small constant to be chosen such that Lemma \ref{lem_nsi} applies. From Lemma~\ref{lem_appr} we can write
\begin{equation}
|\bx( z_i)-\bx( z_j)|^{-s} = | z_i- z_j|^{-s}(1+\calO(( z_i- z_j)^2)),
\end{equation}
for $j=i_L,\dots,i_R$ with $j\ne i$. Since $ z_{j+1}- z_j \ge \delta$ for all $j$, we have
\begin{equation}
| z_i- z_j| \ge |j-i|\delta.
\end{equation}
For those $j$ with $d(z_i, z_j) \ge r_0$, Lemma \ref{lem_nsi} gives $|\bx( z_i)-\bx( z_j)|\ge r_0/2$.  Therefore
\begin{equation}\begin{split}
s\sum_{j\ne i}W(\bx( z_i)&-\bx( z_j)) \\
 \le & \sum_{i_L\le j \le i_R,\,j\ne i}| z_i- z_j|^{-s}(1+\calO(( z_i- z_j)^2)) + CNr_0^{-s} \\
\le & (1+\calO(r_0^2))\sum_{i_L\le j \le i_R,\,j\ne i}(|j-i|\delta)^{-s}  + CNr_0^{-s}\\
\le & (1+\calO(r_0^2))2\cs\delta^{-s}  + CNr_0^{-s}.
\end{split}\end{equation}
Summing over $i$, this gives
\begin{equation}
\begin{split}
E(\bZ) & \le (1+\calO(r_0^2))\tcs N^{-s}\delta^{-s}  + CN^{1-s}r_0^{-s} \\
 & = (1+\calO(r_0^2))\tcs \rho_M^{s}  + CN^{1-s}r_0^{-s},
 \end{split}
\end{equation}
where $\rho_M$ is defined in \eqref{delta}.
We first take $r_0$ small enough so that $r_0^2 \le c\epsilon$, and then $N$ large enough so that $CN^{1-s}r_0^{-s}\le \epsilon$, and the conclusion is obtained (since $\rho_M\ge 1$).

Finally,  inequalities \eqref{EszX} and \eqref{EstLB} proved later in Section~\ref{sectEnergy}  imply that the left-hand inequality in \eqref{lem_E_1}   holds for $N$ for sufficiently large.
\end{proof}

\section{Control on the closest pair}

In this section we analyze the evolution of the closest pairwise distance $\delta$ as defined in \eqref{delta}.
We first give an unconditional lower bound of $\frac{\rd}{\rd{t}}\delta$.
\begin{lemma}\label{lem_close0}
There holds
\begin{equation}\label{lem_close0_1}
\frac{\rd}{\rd{t}}\delta \ge -CN^{-s} N_*\delta^{-s+2} ,\quad N_* := \left\{\begin{split}
& 1,\quad s>2; \\
& \log N,\quad s=2; \\
& N^{-s+2},\quad 1<s<2/ \\
\end{split}\right.
\end{equation}
\end{lemma}

\begin{proof}

We first compute the time derivative of $\delta$:
\begin{equation}\label{zjzim}\begin{split}
 N^{s}\frac{\rd}{\rd{t}}( z_{i_M +1} &- z_{i_M }) 
 \\
  &  =  -\sum_{j\ne i_M +1} \nabla W(\bx( z_{i_M +1})-\bx( z_j)) \cdot \bx'( z_{i_M +1}) \\
 & \qquad  + \sum_{j\ne i_M } \nabla W(\bx( z_{i_M })-\bx( z_j)) \cdot \bx'( z_{i_M })\\
& =  \nabla W(\bx( z_{i_M})-\bx(z_{i_M+1}))\cdot \bx'( z_{i_M}) \\
 & \qquad + \nabla W(\bx( z_{i_M})-\bx(z_{i_M+1}))\cdot \bx'( z_{i_M+1}) \\
& \qquad + \sum_{j \ne i_M,i_M+1}\Big(\nabla W(\bx( z_{i_M})-\bx(z_j))\cdot \bx'( z_{i_M})\\
 & \hspace*{3.1cm}  - \nabla W(\bx( z_{i_M+1})-\bx(z_j))\cdot \bx'( z_{i_M+1})\Big). \\
\end{split}\end{equation}
See Figure \ref{fig2} left as an illustration.

\begin{figure}
\begin{center}
  \includegraphics[width=.49\linewidth]{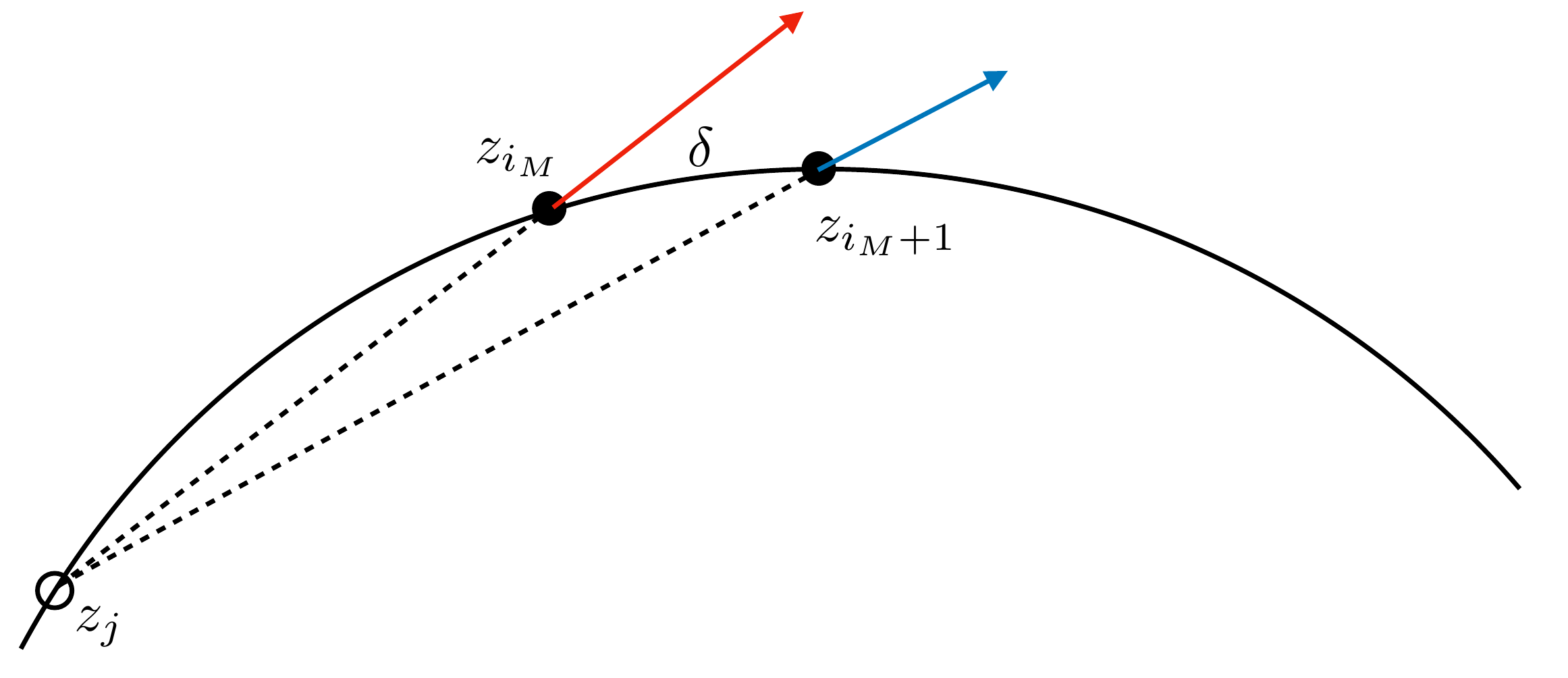}
  \includegraphics[width=.49\linewidth]{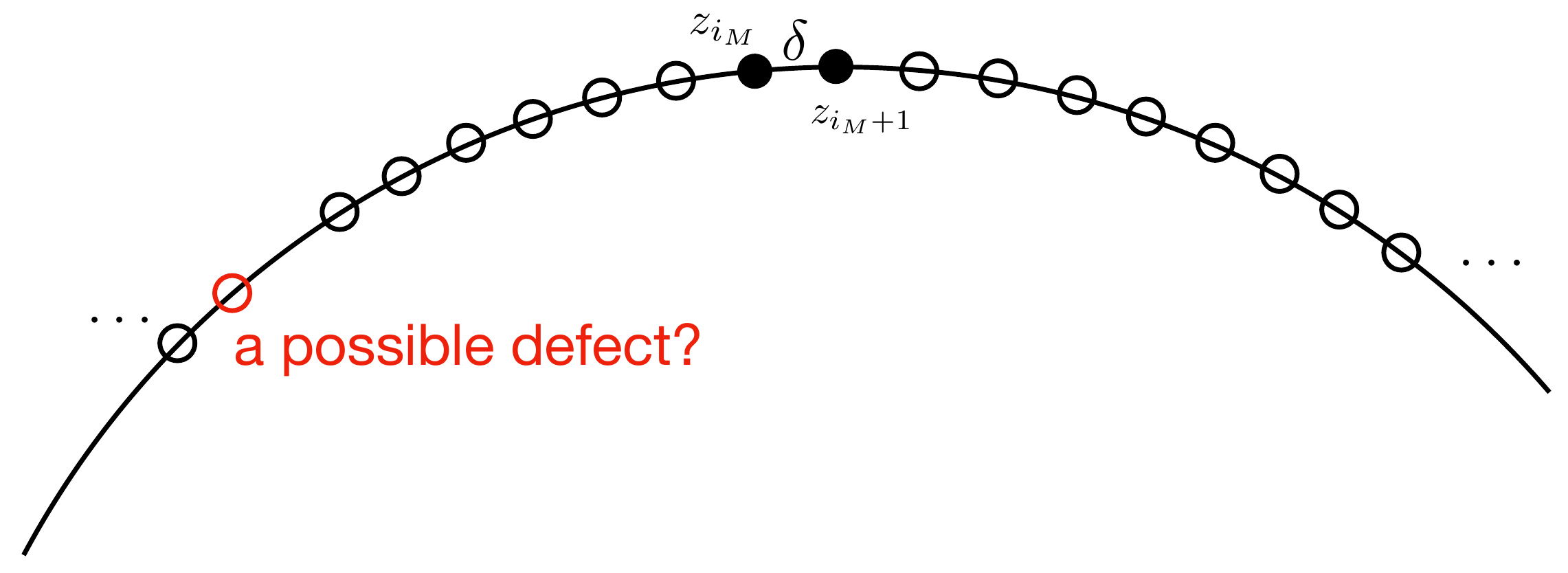}
  \caption{Lemmas \ref{lem_close0} and \ref{lem_close}. Left: the summand in the last term of \eqref{zjzim}. The two terms representing the forces from $z_j$ acting on $z_{i_M}$ (red) and $z_{i_M+1}$ (blue), which decreases/increases $\delta$ respectively. Right: a local uniform distribution like $\{\tilde{z}_j\}$ makes $\frac{\rd}{\rd{t}}\delta \approx 0$ up to errors from curvature. A possible defect will release the total pushing force on $\delta$, make $\frac{\rd}{\rd{t}}\delta$ positive, and thus violate \eqref{lem_close_1}.}
\label{fig2}
\end{center}
\end{figure}

Now we estimate the summand in the last term of \eqref{zjzim} for each $j$, see Figure \ref{fig2} top for an illustration. First notice that if $ d(z,z_{i_M}) \ge r_0$ and $ d(z,z_{i_M+1}) \ge r_0$, then Lemma \ref{lem_nsi} implies that $|\bx( z)-\bx(u)|$ is uniformly bounded below by some $r_1>0$ for any $ z_{i_M}\le u \le  z_{i_M+1}$. Then
\begin{equation}\label{pairest1}\begin{split}
 |\nabla W(\bx( z_{i_M})& -\bx( z))\cdot \bx'( z_{i_M}) - \nabla W(\bx( z_{i_M+1})-\bx( z))\cdot \bx'( z_{i_M+1})| \\
& =  \Big| \int_{ z_{i_M}}^{ z_{i_M+1}} \Big(\bx'(u)^T \nabla^2W(\bx(u)-\bx( z))\bx'(u) \\
 & \hspace*{2.5cm} + \nabla W(\bx(u)-\bx( z))\cdot \bx''(u)\Big)\rd{u} \Big| \\
& \le  C\delta,\qquad \forall z \text{ with }d(z,z_{i_M}) \ge r_0,\, d(z,z_{i_M+1}) \ge r_0.
\end{split}\end{equation}

Then we deal with the case $ z\in ( z_{i_M }-r_0, z_{i_M })$. In view of \eqref{lem_appr_11}, we need to estimate the following quantity:
\begin{equation}\begin{split}
-\phi( z) & :=  W'( z_{i_M }- z)(1+\kappa(z_{i_M})|z_{i_M}-z|^2) \\
 & \qquad - W'( z_{i_M +1}- z)(1+\kappa(z_{i_M})|z_{i_M+1}-z|^2) \\
& =  \Big(| z_{i_M }- z|^{-s-1} + \kappa(z_{i_M})| z_{i_M }- z|^{-s+1}\Big)\\
 & \qquad - \Big(| z_{i_M+1 }- z|^{-s-1} + \kappa(z_{i_M})| z_{i_M+1 }- z|^{-s+1}\Big)
\end{split}\end{equation}
whose derivative can be expressed as
\begin{equation}\label{dphi}
\begin{split}
\phi'( z) &= \psi(z_{i_M+1},z)-  \psi(z_{i_M},z),\\
 & \quad \psi(y,z):=(-s-1)|y-z|^{-s-2} + \kappa(z_{i_M})(-s+1)|y- z|^{-s}.
 \end{split}
\end{equation}
Notice that
\begin{equation}\label{dpsi}\begin{split}
\partial_y\psi(y,z) = & (s+1)(s+2)|y-z|^{-s-3} + \kappa(z_{i_M})(s-1)s|y- z|^{-s-1} \\
= & |y-z|^{-s-3}\Big((s+1)(s+2) - \kappa(z_{i_M})(s-1)s |y-z|^2\Big) > 0
\end{split}\end{equation}
if $|y-z|$ is small. Thus
$\phi'(z) > 0$
since $r_0< z_{i_M} < z_{i_M+1}$ and all three points are within a distance of $r_0+\delta \le r_0+\frac{1}{N}$ which is small.

 Let $i_L,\dots,i_R$ be the indices $j$ with $|z_{i_M}- z_j| \le r_0$. Define the uniform configuration with spacing $\delta$:
\begin{equation*}
\tilde{ z}_j :=  z_{i_M } - (i_M -j)\delta ,\quad i_L\le j \le i_M -1
\end{equation*}
and notice that $ z_j\le \tilde{ z}_j$ by definition of $i_M $. With $I_j:=\int_{ z_j}^{\tilde{ z}_j}\phi'( z)\rd{ z}$, we have
\begin{equation}\label{pairest2}  \begin{split} 
 \sum_{j=i_L}^{i_M -1}&\Big(W'( z_{i_M }- z_j)(1+\kappa(z_{i_M})|z_{i_M} -z_j|^2) \\
  & \qquad - W'( z_{i_M +1}- z_j)(1+\kappa(z_{i_M})|z_{i_M+1}-z_j|^2)\Big) \\
& =   \sum_{j=i_L}^{i_M -1}\left( \Big(W'( z_{i_M }-\tilde{ z}_j)(1+\kappa(z_{i_M})|z_{i_M}-\tilde{z}_j|^2)\right.\\
 &  \qquad \qquad \left. - W'( z_{i_M +1}-\tilde{ z}_j)(1+\kappa(z_{i_M})|z_{i_M+1}-\tilde{z}_j|^2)\Big) + I_j\right)\\
& =   \sum_{j=i_L}^{i_M -1}\left( \Big(W'( (i_M-j)\delta)(1+\kappa(z_{i_M})|(i_M-j)\delta|^2) \right.\\
 & \qquad \qquad \left. - W'( (i_M+1-j)\delta)(1+\kappa(z_{i_M})|(i_M+1-j)\delta|^2)\Big) + I_j\right)\\
& =    W'(\delta)(1+\kappa(z_{i_M})\delta^2) \\
 & \qquad - W'( (i_M+1-i_L)\delta)(1+\kappa(z_{i_M})|(i_M+1-i_L)\delta|^2)  + \sum_{j=i_L}^{i_M -1}I_j \\
= &  -\delta^{-s-1}(1-|i_M +1-i_L|^{-s-1}) \\
 & \qquad -\delta^{-s+1}\kappa(z_{i_M})(1-|i_M +1-i_L|^{-s+1}) + \sum_{j=i_L}^{i_M -1}I_j,
\end{split}\end{equation}
where the third equality follows from a telescoping summation.

Now we have \eqref{pairest1} (together with a similar equality for $i_M +2,\dots,i_R$) and \eqref{pairest2} for the RHS of \eqref{zjzim}.  Combining with \eqref{lem_appr_1} and \eqref{lem_appr_11}, we get
\begin{equation}\label{close1}\begin{split}
 N^{s}\frac{\rd}{\rd{t}}( z_{i_M +1}& - z_{i_M }) \\
& =  2\delta^{-s-1}(1+\kappa(z_{i_M})\delta^2) + \calO(\delta^{-s+2}) \\
& \quad + \sum_{\substack{i_L\le j \le i_R\\  j\ne i_M ,\,i_M +1}}\Big[W'( z_{i_M }- z_j)(1+\kappa(z_{i_M})|z_{i_M}-z_j|^2) \\
 & \quad - W'( z_{i_M +1}- z_j)(1+\kappa(z_{i_M})|z_{i_M+1}-z_j|^2) \\
&  \quad + \calO(( z_{i_M +1}- z_{i_M })(|j-i_M |\delta)^{-s+1})\Big]  + \calO(N\delta) \\
& =  2\delta^{-s-1}(1+\kappa(z_{i_M})\delta^2) + O(\delta^{-s+2}) \\
& \quad  \qquad \Big[-\delta^{-s-1}(1-|i_M +1-i_L|^{-s-1})\\
 & \qquad   \quad -\delta^{-s+1}\kappa(z_{i_M})(1-|i_M +1-i_L|^{-s+1}) + \sum_{j=i_L}^{i_M -1}I_j  \Big] \\
& \qquad + \Big[-\delta^{-s-1}(1-|i_R-i_M|^{-s-1}) \\
 & \qquad \quad -\delta^{-s+1}\kappa(z_{i_M})(1-|i_R-i_M|^{-s+1}) + \sum_{j=i_M +2}^{i_R}I_j \Big] \\
& \qquad + \calO\Big(\delta^{-s+2}\sum_{j=1}^N j^{-s+1}\Big) + O(N\delta) \\
& =  \delta^{-s-1}(|i_M +1-i_L|^{-s-1} + |i_R -i_M|^{-s-1}) \\
 & \qquad + \delta^{-s+1}\kappa(z_{i_M})(|i_M +1-i_L|^{-s+1} + |i_R -i_M|^{-s+1}) \\
& \qquad + \sum_{j=i_L}^{i_M -1}I_j+ \sum_{j=i_M +2}^{i_R} I_j + {\calO}(\delta^{-s+2}N_*)  + \calO(N\delta), \\
\end{split}\end{equation}
where $N_*\sim \sum_{j=1}^N j^{-s+1}$ is defined in \eqref{lem_close0_1}. In the last expression of \eqref{close1}, we can absorb the second term by the first term, using
\begin{equation}\label{absorb}
\delta|i_M+1-i_L| \le \delta\cdot \frac{r_0}{\delta} \le r_0
\end{equation}
and the smallness of $r_0$. The two integrals of $\phi'$ are positive. Therefore
\begin{equation}\begin{split}
& N^{s}\frac{\rd}{\rd{t}}( z_{i_M +1}- z_{i_M }) \ge -C (N_*\delta^{-s+2}  + N\delta).  \\
\end{split}\end{equation}
Then \eqref{lem_close0_1} follows directly by $N\delta\le CN_*\delta^{-s+2}$ which can be easily checked in all three cases, using $N\delta \le 1$.
\end{proof}

Next we state the following lemma: either $\delta(t)$ is increasing very fast, or at $i_M$ the total repulsion is as large as that of a uniform distribution of particles with spacing $\delta(t)$, which is approximately the RHS of \eqref{lem_close_2}.
\begin{lemma}\label{lem_close}
 Fix $\epsilon>0$. For $N>N_0(\epsilon)$, if
\begin{equation}\label{lem_close_1}
\frac{\rd}{\rd{t}}\delta \le 1,
\end{equation}
then
\begin{equation}\label{lem_close_2}
\sum_{i=i_L}^{i_M }\sum_{j=i_M +1}^{i_R}| z_i- z_j|^{-s-1} \ge \cs \delta^{-s-1}(1-\epsilon),
\end{equation}
where $i_L,\dots,i_M -1$ are the indices of particles $ z_i\in ( z_{i_M }-r_0, z_{i_M })$, and $i_M +2,\dots,i_R$ are the indices of particles $ z_i\in ( z_{i_M +1}, z_{i_M +1}+r_0)$.
\end{lemma}

\begin{proof}

We will use the same notations as the previous proof. We claim that for any fixed $J$, there exists $N_0(\epsilon,J)$ such that, $N>N_0$ and $|i_M +2-j|\le J$ imply
\begin{equation}\label{close_1_1}
\tilde{ z}_j- z_j \le \epsilon\delta, \quad \forall j=i_L,\dots,i_M -1 \text{ with } |i_M +2-j| \le J
\end{equation}
under the condition \eqref{lem_close_1}, see Figure \ref{fig2} right for an illustration.

Suppose on the contrary that $\tilde{ z}_j- z_j > \epsilon\delta$ for some $j$ in the range as in \eqref{close_1_1}. Then by \eqref{dphi} and \eqref{dpsi}, for any $z\in [\tilde{z}_{j-1},\tilde{z}_j]$,
\begin{equation}\begin{split}
\phi'(z) & =  \int_{z_{i_M} }^{z_{i_M +1}} \partial_y\psi(y,z)\rd{y}\\
& =  \int_{z_{i_M} }^{z_{i_M +1}} |y-z|^{-s-3}\Big((s+1)(s+2) - \kappa(z_{i_M})(s-1)s |y-z|^2\Big)\rd{y} \\
& \ge  c\int_{z_{i_M} }^{z_{i_M +1}} |y-\tilde{z}_j|^{-s-3}\rd{y}
\ge  c\delta|z_{i_M +1}-z|^{-s-3}\\
& \ge   c\delta^{-s-2}|i_M +2-j|^{-s-3},
\end{split}\end{equation}
where in the first inequality the second term in the integrand is absorbed by the first term using the smallness of $|y-z|\le r_0$. Therefore
\begin{equation}\begin{split}
\int_{ z_j}^{\tilde{ z}_j}\phi'( z)\rd{ z} \ge & \int_{\max\{ z_j,\tilde{ z}_{j-1}\}}^{\tilde{ z}_j}\phi'( z)\rd{ z} \ge \min\{\delta,\tilde{ z}_j- z_j\}\phi'(\tilde{ z}_{j-1}) \\
 \ge & \min\{\delta,\tilde{ z}_j- z_j\}\delta^{-s-2}|i_M +2-j|^{-s-3}.
\end{split}\end{equation}
Therefore, if $\tilde{ z}_j- z_j \ge \epsilon \delta$, then
\begin{equation}
\int_{ z_j}^{\tilde{ z}_j}\phi'( z)\rd{ z} \ge c\epsilon \delta^{-s-1}|i_M +2-j|^{-s-3}
\end{equation}
which gives
\begin{equation}\begin{split}
 \frac{\rd}{\rd{t}}( z_{i_M +1}- z_{i_M }) \ge &  N^{-s}\Big(c\epsilon \delta^{-s-1}|i_M +2-j|^{-s-3}  + \calO(N_*\delta^{-s+2})  + O(N\delta)\Big) \\
= &  c\epsilon (N\delta)^{-s-1}|i_M +2-j|^{-s-3} N + \calO((N\delta)^{-s}N_*\delta^{2})  + \calO(N^{-s}(N\delta)) \\
\end{split}\end{equation}
in view of \eqref{close1}. Notice that $N\delta\le 1$, $N_*\delta \le 1$, and $|i_M +2-j|^{-s-3} \ge J^{-s-3}$. Therefore, by taking $N$ large (in terms of $\epsilon$ and $J$), the first term can absorb the other two terms and gives
\begin{equation}\begin{split}
& \frac{\rd}{\rd{t}}( z_{i_M +1}- z_{i_M }) \ge  c\epsilon (N\delta)^{-s-1}J^{-s-3} N \ge 2 \\
\end{split}\end{equation}
which contradicts \eqref{lem_close_1} if $N$ is large enough. Therefore we proved \eqref{close_1_1}.

Similarly one can show that $ z_j - \tilde{ z}_j \le \epsilon\delta$ for $j=i_M +2,\dots,i_R$ with $|j+1-i_M | \le J$, and also $i_M -i_L \ge J,\,i_R-1-i_M  \ge J$.

Now we aim to show \eqref{lem_close_2}. In fact, \eqref{close_1_1} gives
\begin{equation}\begin{split}
 \sum_{i=i_L}^{i_M }\sum_{j=i_M +1}^{i_R}&| z_i- z_j|^{-s-1} \\
&\ge  \sum_{i=i_M -J+1}^{i_M }\sum_{j=i_M +1}^{i_M +J}|(\tilde{ z}_j+\epsilon\delta)-(\tilde{ z}_i-\epsilon\delta)|^{-s-1} \\
& =  \delta^{-s-1}\sum_{i=i_M -J+1}^{i_M }\sum_{j=i_M +1}^{i_M +J}|j-i + 2\epsilon|^{-s-1} \\
& \ge  \delta^{-s-1}\sum_{i=i_M -J+1}^{i_M }\sum_{j=i_M +1}^{i_M +J}(|j-i|^{-s-1} - (s+1)|j-i|^{-s-2}2\epsilon) \\
& \ge  \delta^{-s-1}\left(\sum_{i=i_M -J+1}^{i_M }\sum_{j=i_M +1}^{i_M +J}|j-i|^{-s-1} - C\epsilon\right) \\
\end{split}\end{equation}
where in the second inequality we used the convexity of the function $x\mapsto |x|^{-s-1}$, and in the third inequality we used the convergence of the series $\sum_{i=-\infty}^{i_M }\sum_{j=i_M +1}^{\infty}|j-i|^{-s-2}$. Since $\sum_{i=-\infty}^{i_M }\sum_{j=i_M +1}^{\infty}|j-i|^{-s-1}=\cs $, one can take $J=J(\epsilon)$ large enough so that $$\sum_{i=i_M -J+1}^{i_M }\sum_{j=i_M +1}^{i_M +J}|j-i|^{-s-1} \ge \cs -\epsilon,$$ and then \eqref{lem_close_2} follows.

\end{proof}

\section{Proof of Theorem \ref{thm1}}


\begin{proof}[Proof of Theorem \ref{thm1}]

{\bf STEP 1}: We aim to give a {\it positive} lower bound
\begin{equation}\label{iMiS}
\sum_{i_M+1\le i \le i_S}\dot{ z}_i \ge \lambda(\rho_M) N
\end{equation}
(where $\rho_M$ is defined in \eqref{delta}) under the assumption \eqref{lem_close_1}, where
\begin{equation}
\lambda(\rho_M) = \left\{\begin{split}
& c(\rho_M-1-\epsilon),\quad \rho_M \le 2 \\
& c\rho_M^{s+1} \\
\end{split}\right.
\end{equation}
for some indices $i_M$ and $i_S$. Notice that  the assumption \eqref{lem_close_1} is equivalent to
\begin{equation}\label{iMiSas}
\frac{\rd}{\rd{t}}\rho_M \ge -N^{-1}\delta^{-2} = -N\rho_M^2
\end{equation}
since $\rho_M = \frac{1}{N\delta}$, see Figure \ref{fig3} for an illustration.

\begin{figure}
\begin{center}
  \includegraphics[width=.46\linewidth]{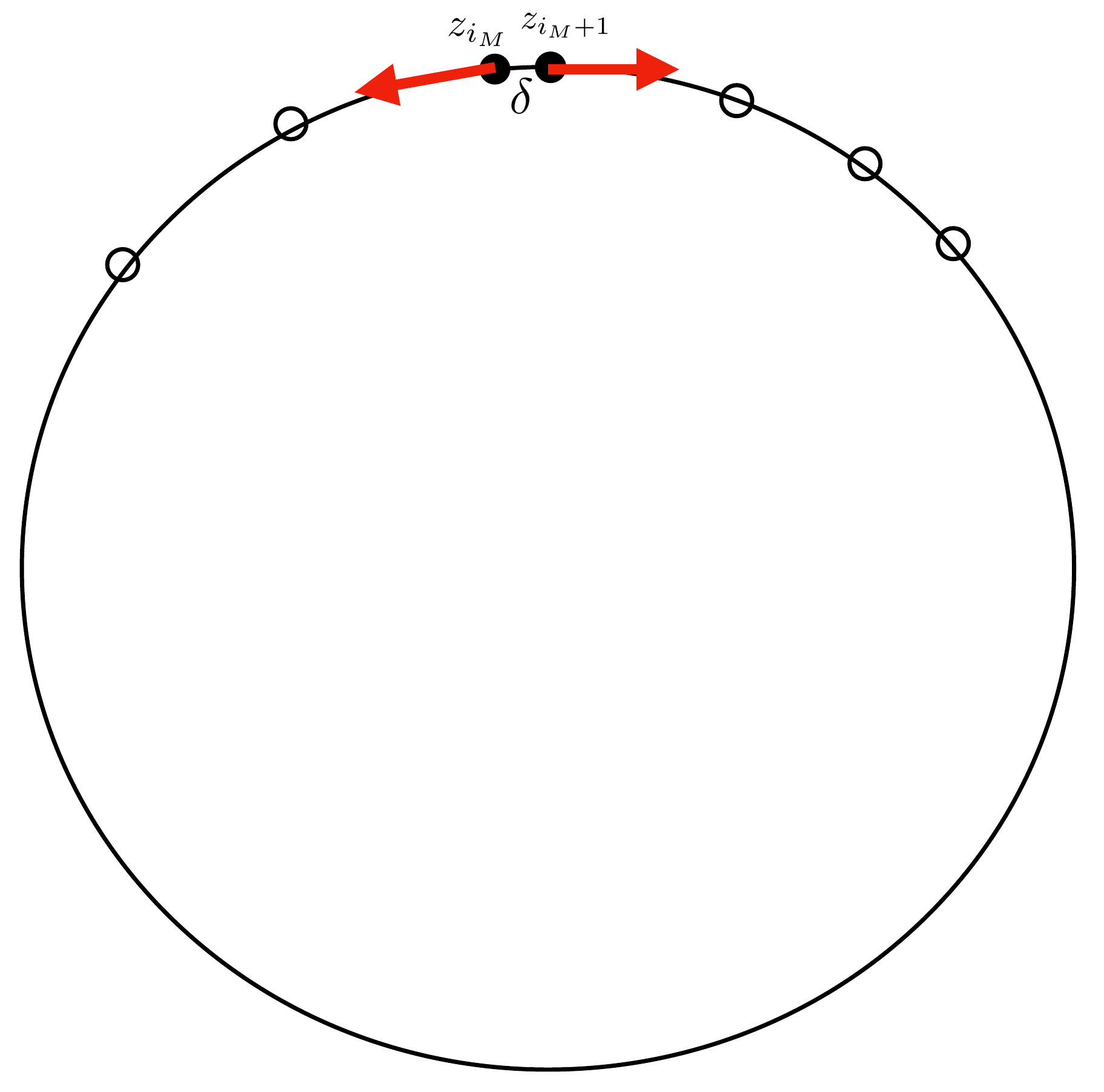}
  \includegraphics[width=.53\linewidth]{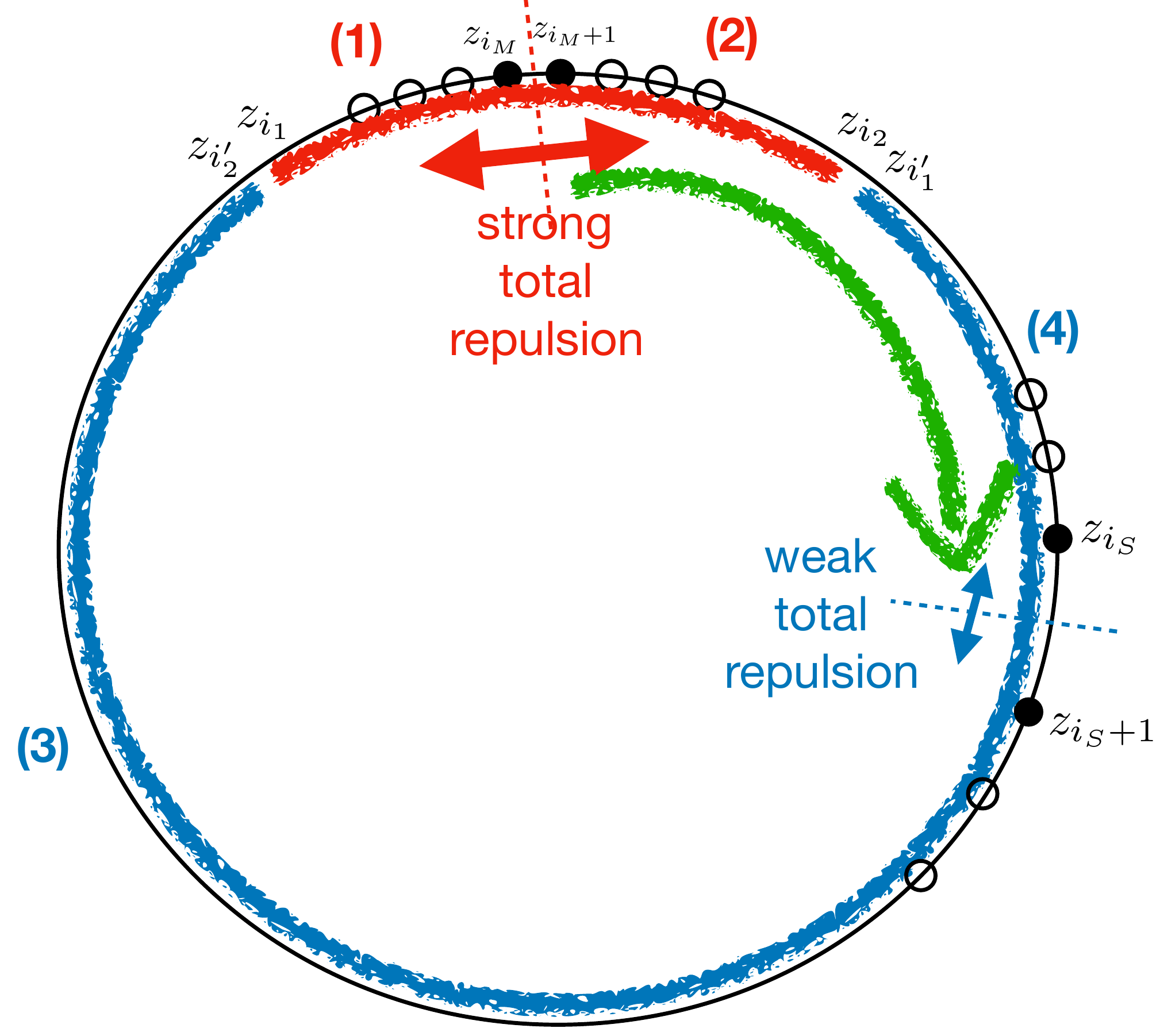}
  \caption{Proof of Theorem \ref{thm1}. Left: when \eqref{lem_close_1} does not hold, $\delta$ is increasing very fast (i.e., $\rho_M$ is decreasing very fast).
  Right: when \eqref{lem_close_1} holds, there is almost uniform distribution near $z_{i_M}$ (red parts) with average density near $\rho_M$, and the total repulsion at $z_{i_M}$ is strong (see \eqref{lem_close_2}). The rest part has average density at most $1+\epsilon$, and Lemma \ref{lem_M2} applies to give a weak total repulsion cut.
  The strong/weak total repulsion ((1)-(2) good contribution, $I_1$, and (3)-(4) bad contribution, $I_2$, see \eqref{sumsi}) forces the green part to rotate. The parameter $r_1$ is to guarantee that (3) or (4) cannot be too short, so that the possible bad contribution from (1)-(4) or (2)-(3) (the term $I_3$) can be neglected.
  }
\label{fig3}
\end{center}
\end{figure}

Using the same notation as in the proof of Lemma \ref{lem_close} (with the same choice of $J$), we have \eqref{close_1_1} from \eqref{lem_close_1}. We take $i_1=i_M -J$ and $i_2=i_M +J$. Then we have
\begin{equation}\label{myI1}
\sum_{i_1\le i\le i_M  < j \le i_2}( z_j-   z_i)^{-s-1} \ge  \cs \delta^{-s-1}(1-\frac{\epsilon}{100}) = \cs (1-\frac{\epsilon}{100})(N\rho_M)^{s+1}.
\end{equation}
Then we take $i_1' = i_2+1$ and $i_2' = i_1+N-1$, which satisfy
\begin{equation}
i_2'- i_1' \ge N - 2J - 2 \ge \frac{N}{2}
\end{equation}
if $N$ is large. Also, by \eqref{close_1_1} we have $ z_{i_1'}- z_{i_M} \le (J+1+\epsilon)\delta \le \frac{J+1+\epsilon}{N}$ and $ z_{i_M}- (z_{i_2'}-1) \le \frac{J+1+\epsilon}{N}$, which implies
\begin{equation}
 z_{i_2'}- z_{i_1'} \ge 1-\frac{2(J+1+\epsilon)}{N} \ge 1-\frac{\epsilon}{100}
\end{equation}
if $N$ is large.

Then Lemma \ref{lem_M2} (with suitable rescaling) applied to $i_1',\dots,i_2'$ gives: there exists an index $i_S$ such that
\begin{equation}
\sum_{i_1'\le i\le i_S < j \le i_2'}( z_j-  z_i)^{-s-1} \le (1+\frac{\epsilon}{100})\cs N^{s+1}
\end{equation}
and
\begin{equation}\label{sdist}
( z_{i_S}, z_{i_S+1})\cap ( z_{i_1'}+r_1, z_{i_2'}-r_1) \ne \emptyset, \quad r_1 =  \frac{\epsilon}{600(s+1)}.
\end{equation}

Now we prove \eqref{iMiS}.
\begin{equation}\label{sumsi}\begin{split}
\sum_{i_M+1\le i \le i_S}\dot{ z}_i = & -N^{-s}\sum_{i_M+1\le i \le i_S}\sum_{i_S+1\le j \le i_M+N} \nabla W(\bx( z_i)-\bx( z_j)) \cdot \bx'( z_i) \\
= & -N^{-s}\Big(\sum_{i_1\le j\le i_M < i \le i_2} + \sum_{i_1'\le i\le i_S < j \le i_2'} + \sum_{\text{others}}\Big) \\
=: & N^{-s}(I_1+I_2+I_3).
\end{split}\end{equation}

Every term in $I_1$ satisfies $ 0<z_i-  z_j \le 2J\delta \le \frac{2J}{N}$ which is small. Thus by applying \eqref{lem_appr_1},
\begin{equation}
I_1 \ge \sum_{i_1\le j\le i_M < i \le i_2}(( z_i-  z_j)^{-s-1} + R_{1,ij})
\end{equation}
with $|R_{1,ij}| \le C_R( z_i-  z_j)^{-s+1}$.

For the terms in $I_2$, if $z_j-z_i> \frac{1}{2}$ and $d(z_i,z_j)<r_0$, then $ d(z_i,z_j) = d(z_i,z_j-1) = z_i-(z_j-1)$, and then by \eqref{lem_appr_2} applied to $z_i$ and $(z_j-1)$,  we have $\nabla W(\bx( z_i)-\bx( z_j)) \cdot \bx'( z_i)<0$ which makes its contribution in \eqref{sumsi} positive. If $| z_j- z_i|>r_0$ then $|\nabla W(\bx( z_i)-\bx( z_j)) \cdot \bx'( z_i)|\le C$ by \eqref{lem_appr_3}. Combined with a similar argument as above for the case $ d(z_i,z_j) =  z_j- z_i<r_0$, we get
\begin{equation}
I_2 \ge -\sum_{i_1'\le i\le i_S < j \le i_2'}(( z_j-  z_i)^{-s-1} + R_{2,ij}) - CN^2
\end{equation}
with $|R_{2,ij}| \le C_R( z_j-  z_i)^{-s+1}$.

We first bound $I_1$ from below. In fact, there exists $C=C(\epsilon)$ such that
\begin{equation}\label{Rij1}
|R_{1,ij}| \le C_R( z_i-  z_j)^{-s+1} \le \frac{\epsilon}{100}( z_i-  z_j)^{-s-1} + C(\epsilon).
\end{equation}
Combining with \eqref{myI1} we get
\[
I_1 \ge \sum_{i_1\le j\le i_M < i \le i_2}(1-\frac{\epsilon}{100})( z_i-  z_j)^{-s-1}-C(\epsilon)N^2\ge (1-\frac{\epsilon}{100})^2 \cs (N\rho_M)^{s+1} - C(\epsilon)N^2.
\]
Similarly
\[
I_2 \ge -(1+\frac{\epsilon}{100})^2 \cs (N(1+\epsilon))^{s+1} - C(\epsilon)N^2.
\]

To bound $I_3$, we recall the definition of $r_1$ in \eqref{sdist}. We notice that for $i \in [i_M+1,i_S]$ and $j\in [i_S+1,i_M+N]$, if $d(z_i,z_j) \le r_1$ and $\nabla W(\bx( z_i)-\bx( z_j)) \cdot \bx'( z_i)>0$, then by \eqref{lem_appr_11} one necessarily has $ z_j\in [ z_i, z_i+r_1]$. The only possibility for this to happen is when $ z_i\in [ z_{i_S+1}-r_1, z_{i_S+1}]$ and $ z_j\in [ z_{i_S}, z_{i_S}+r_1]$. But by \eqref{sdist}, $[ z_{i_S+1}-r_1, z_{i_S+1}] \subset [z_{i_2},z_{i_S+1}]$ and $[ z_{i_S}, z_{i_S}+r_1]\subset [z_{i_S}, z_{i_1'}]$, and thus the term in \eqref{sumsi} with indices $(i,j)$ is already included in $I_2$. Therefore, every term in $I_3$ has either $d( z_i,z_j) > r_1$ or $\nabla W(\bx( z_i)-\bx( z_j)) \cdot \bx'( z_i)\le 0$, and thus
\[
I_3 \ge -C(\epsilon)N^2
\]
by \eqref{lem_appr_3} (where the $\epsilon$-dependence comes from that of $r_1$).

In conclusion, we get
\[
\sum_{i_M+1\le i \le i_S}\dot{ z}_i \ge \left( (1-\frac{\epsilon}{100})^2\rho_M^{s+1} -(1+\frac{\epsilon}{100})^2(1+\epsilon)^{s+1} \right)\cs N - C(\epsilon)N^{-s+2}.
\]
Now we show that the quantity in the big parenthesis above is bounded below. In fact, using $\epsilon<1$,
\[
\begin{split}
 (1-\frac{\epsilon}{100})^2\rho_M^{s+1} &-(1+\frac{\epsilon}{100})^2(1+\epsilon)^{s+1} \\
\ge &\frac{1}{2}(\rho_M-1-2\epsilon) +  (1-\frac{\epsilon}{100})^2(1+2\epsilon)^{s+1} -(1+\frac{\epsilon}{100})^2(1+\epsilon)^{s+1} \\
\ge &\frac{1}{2}(\rho_M-1-2\epsilon) +  (1+\epsilon)^{s+1}((1-\frac{\epsilon}{100})^2(1+\frac{\epsilon}{2}) - (1+\frac{\epsilon}{100})^2) \\
\ge &\frac{1}{2}(\rho_M-1-2\epsilon). \\
\end{split}
\]

Therefore, we get
\[
\sum_{i_M+1\le i \le i_S}\dot{ z}_i \ge \frac{1}{2}(\rho_M-1-2\epsilon) \cs N - C(\epsilon)N^{-s+2} \ge \frac{1}{4}(\rho_M-1-3\epsilon)\cs N
\]
 if $N$ is large. Also, if $\rho_M\ge 2$, then
\[
\begin{split}
(1-\frac{\epsilon}{100})^2\rho_M^{s+1}& -(1+\frac{\epsilon}{100})^2(1+\epsilon)^{s+1} \\
 & \ge \frac{1}{4}\rho_M^{s+1} + 2^{s}(1-\frac{\epsilon}{100})^2 -(1+\frac{\epsilon}{100})^2(1+\epsilon)^{s+1} \ge \frac{1}{4}\rho_M^{s+1}
\end{split}
\]
and we get
\begin{equation}\label{iMiS2}
\sum_{i_M+1\le i \le i_S}\dot{ z}_i \ge c\rho_M^{s+1}N
\end{equation}
 if $N$ is large.

{\bf STEP 2:} We use \eqref{iMiS} (under the condition \eqref{iMiSas}) to give energy dissipation rate, and use it to define a Lyapunov-like functional.

If $\rho_M-1-\epsilon\ge 0$, then Cauchy-Schwarz gives
\[
c^2(\rho_M-1-\epsilon)^2N^2 \le \Big(\sum_{i_M+1\le i \le i_S}\dot{ z}_i\Big)^2 \le (i_S- i_M)\sum_{i_M+1\le i \le i_S}|\dot{ z}_i|^2 \le N \sum_i|\dot{ z}_i|^2.
\]
Recalling the energy dissipation law \eqref{eq0E}, we get
\begin{equation}\label{dE1}
\frac{\rd}{\rd{t}}E(t) \le -c^2((\rho_M-1-\epsilon)_{\ge 0})^2, \quad \text{if }\frac{\rd}{\rd{t}}\rho_M \ge -N\rho_M^2
\end{equation}
and similarly
\begin{equation}\label{dE2}
\frac{\rd}{\rd{t}}E(t) \le -c^2\rho_M^{2(s+1)},\quad \text{if }\frac{\rd}{\rd{t}}\rho_M \ge -N\rho_M^2,\quad \rho_M \ge 2.
\end{equation}

Since $\rho_M = \frac{1}{N\delta}$, Lemma \ref{lem_close0} gives
\begin{equation}\label{drhoM}
\frac{\rd}{\rd{t}} \rho_M = -\frac{1}{N\delta^2}\cdot \frac{\rd}{\rd{t}}\delta \le CN^{-1}\delta^{-2}\cdot N^{-s}N_*\delta^{-s+2} = \frac{CN_*}{N}\rho_M^s.
\end{equation}

Define a Lyapunov-like functional
\begin{equation}
F(t) = E(t) + \rho_M(t)^{s}.
\end{equation}
Then at any time $t$ with $\rho_M(t) \ge 1+2\epsilon$, at least one of the following three options must hold:
\begin{itemize}
\item When $\frac{\rd}{\rd{t}}\rho_M < -N\rho_M^2$, using $\frac{\rd}{\rd{t}}E \le 0$,
\begin{equation}\label{opt1}
\frac{\rd}{\rd{t}}F \le -sN\rho_M^{s+1}.
\end{equation}
\item When $\frac{\rd}{\rd{t}}\rho_M \ge -N\rho_M^2$ and $\rho_M \ge 2$, \eqref{dE2} and \eqref{drhoM} give
\begin{equation}\label{opt2}
\frac{\rd}{\rd{t}}F \le -c\rho_M^{2s+2} + \frac{CN_*}{N}\rho_M^{2s-1} \le -c\rho_M^{2s+2}
\end{equation}
by taking $N$ large, since $\rho_M\ge 1$ always holds and $\lim_{N\rightarrow\infty}\frac{N_*}{N} = 0$.
\item When $\frac{\rd}{\rd{t}}\rho_M \ge -N\rho_M^2$ and $ 1+2\epsilon\le \rho_M < A$ (with $A>2$ an absolute constant to be determined), \eqref{dE1} and \eqref{drhoM} give
\begin{equation}\label{opt3}
\frac{\rd}{\rd{t}}F \le -c(\rho_M-1-\epsilon)^2 + \frac{CN_*}{N}\rho_M^{2s-1} \le -c(\rho_M-1-2\epsilon)^2
\end{equation}
by taking $N$ large (which may depend on $A$).
\end{itemize}

{\bf STEP 3}: We use the functional $F$ to give convergence rate of $\rho_M$ to 1 up to an error of $\calO(\epsilon)$.

Let $T_1$ be the first time such that $\rho_M\le 2$, and we aim to estimate $T_1$. For $0\le t \le T_1$, either \eqref{opt1} or \eqref{opt2} happens. Recall that $E\le C\rho_M^s$ from Lemma \ref{lem_E}, and therefore we have
\begin{equation}
\frac{\rd}{\rd{t}}F \le -cF^{\frac{s+1}{s}}.
\end{equation}
Since $\frac{s+1}{s}>1$, there exists an absolute constant $C_{T,1}$ (independent of $F(0)$) such that $F(C_{T,1}) \le 1/2$ if the above ODE holds for $0\le t \le C_{T,1}$, which contradicts the fact that $F\ge 1$. Therefore there must hold
\begin{equation}
T_1\le C_{T,1}.
\end{equation}

Then we have the estimate
\begin{equation}
F(T_1) \le C\rho_M(T_1)^{s} \le C2^{s} =: A^s
\end{equation}
where $A$ is  the constant appeared in the condition of \eqref{opt3}.

Let $T_2$ be the first time such that $\rho_M\le 1+B\epsilon$, where $B>2$ is a positive constant to be determined. For $T_1\le t \le T_2$, if $\rho_M(t) \le A$, then either \eqref{opt1} or \eqref{opt3} happens, and we have
\begin{equation}
\frac{\rd}{\rd{t}}F \le -c(\rho_M-1-2\epsilon)^2.
\end{equation}
This in particular implies $F(t) \le A^s$ for $T_1\le t \le T_2$, which in turn implies the assumption $\rho_M(t) \le A$. Then
\begin{equation}\begin{split}
\rho_M-1-2\epsilon &\ge  c\Big((1+\epsilon)\tcs +1\Big)\Big(\rho_M^{s}-(1+2\epsilon)^{s}\Big) \\
& \ge  c\left[\Big((1+\epsilon)\tcs +1\Big)\Big(\rho_M^{s}-(1+2\epsilon)^{s}\Big) \right. \\
 & \left. \quad + \Big(E-(1+\epsilon)\tcs \rho_M^{s}\Big)\right] \\
& =  c[F - ((1+\epsilon)\tcs +1)(1+2\epsilon)^{s}]
\end{split}\end{equation}
where the second inequality uses Lemma \ref{lem_E}. Therefore $\tilde{F}:=F- ((1+\epsilon)\tcs +1)(1+2\epsilon)^{s}$ satisfies
\begin{equation}
\frac{\rd}{\rd{t}}\tilde{F} \le -c\tilde{F}^2,\quad T_1\le t \le T_2
\end{equation}
which implies
\begin{equation}
\tilde{F}(t) \le \frac{1}{c(t-T_1)+\frac{1}{\tilde{F}(T_1)}} \le \frac{1}{c(t-T_1)+A^{-s}}.
\end{equation}
Therefore if $t-T_1 \ge \frac{C}{\epsilon}$ with $T_1\le t \le T_2$, then $\tilde{F}(t) \le \epsilon$, which implies
\begin{equation}\label{Fcont}
F(t) \le ((1+\epsilon)\tcs +1)(1+2\epsilon)^{s} + \epsilon.
\end{equation}
On the other hand $\rho_M(t)\le 1+B\epsilon$. This together with Lemma \ref{lem_E} implies
\begin{equation}
F(t) \ge (1-\epsilon)\tcs  + (1+B\epsilon)^{s}
\end{equation}
which is a contradiction against \eqref{Fcont} if $B$ is large enough (only depending on $s$). Therefore we get
\begin{equation}
T_2\le \frac{C}{\epsilon}
\end{equation}
and then Lemma \ref{lem_E} gives
\begin{equation}
E(T_2) \le (1+\epsilon)\tcs \rho_M(T_2)^{s} \le (1+\epsilon)\tcs (1+B\epsilon)^{s}  \le (1+C\epsilon)\tcs.
\end{equation}
$E(t)$ also satisfies the last inequality if $t\ge T_2$, since $E(t)$ is non-increasing.

\end{proof}

\section{Energy and distribution}\label{sectEnergy}


Recall that the energy of a configuration paramatrized by $ \bZ$  is
$$
E= E(\bZ) :=\frac{1}{sN^{s+1}}\sum_{1\le i< j\le N}^N |\bx(z_j)-\bx(z_i)|^{-s},$$
 and observe that
\[
\begin{split}
 E(\bZ) &= \frac{1}{2sN^{s+1}}\sum_{i=1}^N\sum_{j=i+1}^{i+N-1}|\bx(z_j)-\bx(z_i)|^{-s}=  \frac{1}{2sN^{s+1}}\sum_{i=1}^N\sum_{k= 1}^{N-1}|\bx(z_{i+k})-\bx(z_i)|^{-s} \\&= \frac{1}{2}\sum_{k= 1}^{N-1}E^k(\bZ),
\end{split}
\]
 where
 $$
 E^k(\bZ):= \frac{1}{sN^{s+1}}\sum_{i=1}^N|\bx(z_{i+k})-\bx(z_i)|^{-s}.
 $$
One may easily verify that $E^k(\bZ)= E^{N-k}(\bZ)$ for $1\le k<N$ and thus
\begin{equation} \label{EsOE} E (\bZ)= \begin{cases}   \sum_{k= 1}^{\frac{N-1}{2}}E ^k(\bZ), & \text{for $N$ odd},\\
 \sum_{k= 1}^{\frac{N}{2}-1}E^k(\bZ) +(1/2)E^{N/2}(\bZ), & \text{for $N$ even.}
\end{cases}
\end{equation}

For $1\le k\le N-1$, we  define
 $$
\tilde{E}^k(\bZ):=  \frac{1}{sN^{s+1}}\sum_{i=1}^N(z_{i+k} - z_i)^{-s},
 $$
 and
 \begin{equation} \label{EstOE} \tilde{E}(\bZ)= \begin{cases}   \sum_{k= 1}^{\frac{N-1}{2}}\tilde{E} ^k(  \bZ), & \text{for $N$ odd},\\
 \sum_{k= 1}^{\frac{N}{2}-1}\tilde{E}^k( \bZ) +(1/2)\tilde{E}^{N/2}(\bZ), & \text{for $N$ even.}
\end{cases}
\end{equation}

 Since $\bx(z)$ is an arc-length parametrization, we have $|\bx(z)-\bx(z')|\le  |z-z'|$ for all $z,z'\in \mathbb{R}$ and thus
 \begin{equation}\label{EszX}
\tilde{E}(\bZ)\le  E(\bZ),
 \end{equation}
for any $\bZ$.
Let
 \begin{equation}\zeta(s;N):= \sum_{k=1}^{\lfloor\frac{N-1}{2}\rfloor}k^{-s}.
\end{equation}

 \begin{lemma}\label{EsLem} For $k,N\in \mathbb{N}$ and $s>0$,
  \begin{equation}\label{EstkLB}
s^{-1}k^{-s}\le \tilde{E}^k(\bZ) \le k^{-s} \tilde{E}^1(\bZ),
\end{equation}
and
\begin{equation}\label{EstLB}
s^{-1}\zeta(s;N) \le \tilde{E}^1(\bZ)+s^{-1}(\zeta(s;N)-1) \le \tilde{E}(\bZ).
\end{equation}
 \end{lemma}
 \begin{proof}
By Jensen's inequality,
\[
\begin{split}
sN^{s+1}\tilde{E}^1(\bZ)&= \sum_{i=1}^N (z_{i+1} - z_i)^{-s}= \frac{1}{k}\sum_{j=0}^{k-1}\sum_{i=1}^N (z_{i+j+1} - z_{i+j})^{-s}\\
&= \sum_{i=1}^N\frac{1}{k}\sum_{j=0}^{k-1} (z_{i+j+1} - z_{i+j})^{-s}\ge sN^{s+1}k^s\tilde{E}^k(\bZ),
\end{split}
\]
and
 $$
\tilde{E}^k(\bZ)= s^{-1}N^{-s}\sum_{i=1}^N(z_{i+k} - z_i)^{-s}\frac{1}{N}\ge s^{-1}\left(\sum_{i=1}^N {(z_{i+k} - z_i}) \right)^{-s}=s^{-1}k^{-s},
$$
proving \eqref{EstkLB}.   From \eqref{EstOE}, it follows that $\tilde{E}(\bZ)\ge \sum_{k=1}^{\lfloor\frac{N-1}{2}\rfloor}\tilde{E}^k(\bZ)$ which together with \eqref{EstkLB} establishes \eqref{EstLB}. \end{proof}

In the next lemma we show that the     mean absolute deviation  of the neighbor arclength distances $d_i:=z_{i+1}-z_i$ is small on the microscopic scale.   As a consequence we derive a macroscopic result showing that the density of points is nearly uniform when $N$ is sufficiently large and the energy is sufficiently close to its minimal value.
\begin{lemma}\label{lem_EZ2}
Let $\epsilon>0$,  $s>1$,   $N\ge 2$, and define
\begin{equation}
\Delta:= 2\left(\frac{2\zeta(s)}{s(s+1)}\right)^{1/2}.
\end{equation}
If  $\bZ=(z_1,z_2,\ldots,z_N)$ satisfies
\begin{equation}\label{EZ2bnd}
\tilde{E}(\bZ) \le s^{-1}\zeta(s;N)(1+\epsilon),
\end{equation}
then the mean absolute deviation of  $d_i:=z_{i+1}-z_i$, $i=1,2,\ldots, N$, satisfies
\begin{equation}\label{Zmad}
\frac{1}{N}\sum_{i=1}^N \left |d_i-\frac{1}{N}\right |\le\frac{\Delta\epsilon^{1/2}}{N}.\end{equation}
\end{lemma}

\begin{proof}
Inequalities \eqref{EstLB} and \eqref{EZ2bnd} imply
\begin{equation}\label{sE1bnd}
s \tilde{E}^1(\bZ)\le 1+ \zeta(s;N)\epsilon.
\end{equation}
We write $\tilde{E}^1(\bZ)$ as
\begin{equation}
\tilde{E}^1(\bZ) = \frac{1}{N^{s+1}}\sum_i W(d_i),\quad W(x):=\frac{x^{-s}}{s}.
\end{equation}
The Taylor expansion of $W$ at $\frac{1}{N}$ gives
\begin{equation}
W(d_i) = W(\frac{1}{N}) + W'(\frac{1}{N}) (d_i-\frac{1}{N}) + \frac{1}{2}W''(\xi_i)(d_i-\frac{1}{N})^2,
\end{equation}
where $\xi_i$ is between $d_i$ and $\frac{1}{N}$. Substituting into the previous equation gives
\begin{equation}\begin{split}
s\tilde{E}^1(\bZ) = & \frac{s}{N^{s+1}}\sum_i \Big(W(\frac{1}{N}) + W'(\frac{1}{N}) (d_i-\frac{1}{N}) + \frac{1}{2}W''(\xi_i)(d_i-\frac{1}{N})^2\Big) \\
= & \frac{s}{N^{s+1}}\sum_iW(\frac{1}{N}) + W'(\frac{1}{N}) \frac{s}{N^{s+1}}\sum_i(d_i-\frac{1}{N}) \\
 & +  \frac{s}{2N^{s+1}}\sum_iW''(\xi_i)(d_i-\frac{1}{N})^2 \\
= & 1  + \frac{1}{2}\cdot\frac{s}{N^{s+1}}\sum_iW''(\xi_i)(d_i-\frac{1}{N})^2, \\
\end{split}\end{equation}
using $\sum_i d_i = 1 = \sum_i \frac{1}{N}$. Combined with \eqref{sE1bnd}, we get
\begin{equation}
\frac{1}{2}\cdot\frac{s}{N^{s+1}}\sum_iW''(\xi_i)(d_i-\frac{1}{N})^2 \le \zeta(s;N)\epsilon.
\end{equation}

Notice that for every $i$ with $d_i< 1/N$, we have $\xi_i \in (d_i,\frac{1}{N})$, and thus
\begin{equation}
W''(\xi_i) = (s+1)\xi_i^{-s-2} \ge (s+1)N^{s+2}.
\end{equation}
Therefore,
\begin{equation}\begin{split}
\frac{1}{N}\sum_{i:\,d_i< 1/N} \left |d_i-\frac{1}{N}\right |\le & \left(\frac{1}{N}\sum_{i:\,d_i< 1/N} \left |d_i-\frac{1}{N}\right |^2\right)^{1/2} \\
 \le & \left(\frac{1}{(s+1)N^{s+3}}\sum_{i:\,d_i< 1/N} W''(\xi_i)\left |d_i-\frac{1}{N}\right |^2\right)^{1/2} \\
\le &  \left(\frac{1}{(s+1)N^{s+3}}\cdot \frac{2N^{s+1}}{s}\zeta(s;N)\epsilon \right)^{1/2} \\
= & \left(\frac{2\zeta(s;N)}{s(s+1)}\right)^{1/2}\frac{\epsilon^{1/2}}{N}.
\end{split}\end{equation}
Combined with the fact that
\begin{equation}
\frac{1}{N}\sum_{i} \left |d_i-\frac{1}{N}\right | = 2\cdot \frac{1}{N}\sum_{i:\,d_i< 1/N} \left |d_i-\frac{1}{N}\right |,
\end{equation}
we obtain the conclusion.
\end{proof}


We next show that the macroscopic density must be nearly uniform when the energy is nearly optimal.
\begin{lemma}\label{lem_EZ3}
Let  $0<\epsilon<1$,  $s>1$,  and  $N\ge 2^{-s+1}(s+1)\epsilon^{-1}$.
If  $\bZ=(z_1,z_2,\ldots,z_N)$ satisfies
\begin{equation}\label{EZ3bnd}
\tilde{E}(\bZ) \le s^{-1}\zeta(s;N)(1+\epsilon),
\end{equation}
then for all $a\in \mathbb{R}$ and   $0<L<1$,
\begin{equation}\label{disc0}
 \left|\frac{\#\{i:\,[z_i,z_{i+1})\subset [a,a+L)\}}{N} -L\right|\le \left[L(1-L)\tcs\right]^{1/2}(2\epsilon)^{1/2} .\end{equation}
\end{lemma}

\begin{proof}

First, we may assume $L\le 1/2$, since one can reduce the case $L>1/2$ to $L\le 1/2$ by replacing $[a,a+L)$ by $[a+L,a+1)$.  

Let  $M:=\#\{i:\,z_i\in [a,b]\}$, $J_1:=\{i\in \mathbb{Z}:a\le z_i<z_{i+1}< b\}$, $J_2:= \{i\in \mathbb{Z}:b\le z_i<z_{i+1}<a+1)\}$,  $N_1:=\# J_1$, $N_2:=\# J_2$, and $\alpha=N_1/N$.  If $0<M<N$, then $N_1=M-1$ and  $N_2=N-M-1$ so that $N_1+N_2=N-2$.  If $M=0$ or $M=N$, then 
$N_1+N_2=N-1$.  Thus, $N-N_1-2\le N_2\le N-N_1-1$.    Using the conditions $\epsilon<1$ and $L\le 1/2$, it is straightforward to show that $N_2$ is always positive for sufficiently large $N$.
We also observe that $\sum_{i\in J_1} d_i \le L$ and $\sum_{i\in J_2} d_i \le 1-L$. Therefore, by Jensen's inequality, when $N_1>0$, 
\[
\begin{split}
s\tilde{E}^1(\bZ) \ge & \frac{1}{N^{s+1}}\sum_{i\in J_1} d_i^{-s} + \frac{1}{N^{s+1}}\sum_{i\in J_2} d_i^{-s} \\
  \ge & \frac{N_1}{N^{s+1}}  \left(\frac{1}{N_1}\sum_{i\in J_1} d_i\right)^{-s} + \frac{N_2}{N^{s+1}} \left(\frac{1}{N_2}\sum_{i\in J_2} d_i\right)^{-s} \\
\ge & \frac{N_1}{N^{s+1}}  \left(\frac{L}{N_1}\right)^{-s} + \frac{N_2}{N^{s+1}}  \left(\frac{1-L}{N_2}\right)^{-s}  \\
  = &\alpha^{s+1}L^{-s}+ (1-\frac{2}{N}-\alpha)^{s+1}(1-L)^{-s} \\
\ge & \alpha^{s+1}L^{-s}+ (1-\alpha)^{s+1}(1-L)^{-s}  - \frac{2(s+1)\cdot 2^{-s}}{N}
\end{split}
\]
and it is clear that the last inequality is also true when $N_1=0$.
 Using now  the convexity of $x\to x^s$, we have
\begin{equation}\label{halLB}\begin{split}
s\tilde{E}^1(\bZ) & + \frac{2(s+1)\cdot 2^{-s}}{N}\\
 &\ge\alpha(\alpha/L)^s+(1-\alpha)((1-\alpha)/(1-L))^s\ge\left(\frac{\alpha^2}{L}+\frac{(1-\alpha)^2}{1-L}\right)^s\\
 &=\left(1+\frac{(\alpha-L)^2}{L(1-L)}\right)^s\ge 1+\frac{s}{L(1-L)}(\alpha-L)^2.
\end{split}
\end{equation}

As in the proof of Lemma~\ref{lem_EZ2}, inequalities \eqref{EstLB} and \eqref{EZ3bnd} imply that \eqref{sE1bnd} holds.  By assumption, $\frac{2(s+1)\cdot 2^{-s}}{N} \le  \epsilon\le \zeta(s;N)\epsilon$. So, in light of
\eqref{halLB}, we obtain
$$
(\alpha-L)^2\le  2\epsilon{\zeta(s;N)L(1-L)}/{s} \le \tcs L(1-L)\cdot 2\epsilon,
$$
which,  gives \eqref{disc0}.
\end{proof}

Theorem~\ref{thm2} follows directly from Lemmas~\ref{lem_EZ2} and \ref{lem_EZ3}.
\begin{proof}[Proof of Theorem~\ref{thm2}]
Let $N_0$ be large enough so that  $(1+\epsilon)\cs/ \zeta(s;N_0)\le (1+ 2\epsilon)$. From \eqref{EszX}, we have
$$
\tilde{E}(\bZ) \le E(\bZ)\le  \tcs (1+\epsilon)\le s^{-1}\zeta(s;N) (1+2\epsilon).
$$
Then Lemma~\ref{lem_EZ2} implies   \eqref{Zmad1}   while Lemma~\ref{lem_EZ3} shows that
\eqref{disc1} holds.
\end{proof}

\bibliographystyle{amsplain}

\end{document}